\documentclass[12pt]{amsart}
\usepackage{amssymb}
\usepackage{amsmath}
\usepackage{amscd}
\usepackage[utf8]{inputenc}
\usepackage[all]{xy}
\usepackage{tikz-cd}
\usepackage{natbib}

\newtheorem{theorem}{Theorem}
\newtheorem{corollary}[theorem]{Corollary}
\newtheorem{proposition}[theorem]{Proposition}
\newtheorem{lemma}[theorem]{Lemma}

\theoremstyle{remark}
\newtheorem{remark}[theorem]{Remark}
\newtheorem{example}[theorem]{Example}

\theoremstyle{definition}
\newtheorem{definition}[theorem]{Definition}

\numberwithin{equation}{subsection}

\DeclareMathOperator{\Symm}{\mathtt{S}}
\DeclareMathOperator{\Hom}{Hom}

\DeclareMathOperator{\CA}{\mathcal{CA}}
\DeclareMathOperator{\ECA}{\mathcal{ECA}}
\DeclareMathOperator{\DirStr}{Dir}
\DeclareMathOperator{\Desc}{Desc}

\DeclareMathOperator{\coker}{coker}

\DeclareMathOperator{\ann}{ann}
\DeclareMathOperator{\Cour}{\mathtt{Q}}
\DeclareMathOperator{\Man}{\mathfrak{Man}}
\DeclareMathOperator{\DGMan}{\mathrm{dg}\mathfrak{Man}}

\newcommand{\opdp}{\dotplus}
\newcommand{\op}{\mathtt{op}}
\newcommand\OExt[1]{{\mathcal{O}[#1]\mathrm{Ext}}}
\newcommand\CExt[1]{\mathrm{\mathcal{C}Ext}({#1})}
\newcommand{\id}{\mathrm{id}}
\newcommand\dual[1]{{#1}^{\vee}}
\newcommand{\pr}{{\mathtt{pr}}}
\newcommand{\ip}{{\langle\ ,\ \rangle}}
\newcommand{\Dirac}{{\mathcal{D}\!\!\!\!/}}
\newcommand{\ev}{\mathtt{ev}}
\newcommand{\ii}{{++}}
\newcommand\Mod[1]{{#1}\!-\!\mathtt{Mod}}
\newcommand\LA[1]{{#1}\!-\!\mathtt{LieAlgd}}

\makeatletter
\newcommand*{\doublerightarrow}[2]{\mathrel{
  \settowidth{\@tempdima}{$\scriptstyle#1$}
  \settowidth{\@tempdimb}{$\scriptstyle#2$}
  \ifdim\@tempdimb>\@tempdima \@tempdima=\@tempdimb\fi
  \mathop{\vcenter{
    \offinterlineskip\ialign{\hbox to\dimexpr\@tempdima+1em{##}\cr
    \rightarrowfill\cr\noalign{\kern.5ex}
    \rightarrowfill\cr}}}\limits^{\!#1}_{\!#2}}}
\newcommand*{\triplerightarrow}[1]{\mathrel{
  \settowidth{\@tempdima}{$\scriptstyle#1$}
  \mathop{\vcenter{
    \offinterlineskip\ialign{\hbox to\dimexpr\@tempdima+1em{##}\cr
    \rightarrowfill\cr\noalign{\kern.5ex}
    \rightarrowfill\cr\noalign{\kern.5ex}
    \rightarrowfill\cr}}}\limits^{\!#1}}}
\makeatother

\begin{document}

\title{Odd transgression for Courant algebroids}

\author{Paul Bressler}
\address{Universidad de los Andes, Bogot\'a, Colombia}
\email{paul.bressler@gmail.com}

\author{Camilo Rengifo}
\address{Universidad de La Sabana, Ch\'ia, Colombia}
\email{camiloregu@unisabana.edu.co}

\maketitle

\begin{abstract}
The ``odd transgression" introduced in \citep{BR} is applied to construct and study the inverse image functor in the theory of Courant algebroids.
\end{abstract}

\section{Introduction}
The goal of this note is to demonstrate applications of ``odd transgression" introduced in \citep{BR} to the theory of Courant algebroids.

The ``odd transgression" functor associates to a Courant algebroid $\mathcal{Q}$ a differential-graded (DG) Lie algebroid $\tau\mathcal{Q}$ over the de Rham complex equipped with a central section of degree -2 which we refer to as a \emph{marking}. Conversely, a marked DG Lie algebroid $\mathcal{A}$ over the de Rham complex satisfying certain natural vanishing conditions gives rise, by way of taking its component in degree -1, to a Courant algebroid $\Cour(\mathcal{A})$. In particular, the Courant algebroid $\mathcal{Q}$ is recovered in this manner from its transgression $\tau\mathcal{Q}$. Using the functors $\tau$ and $\Cour$ one can project various standard constructions in Lie algebroids to the, perhaps less familiar, setting of Courant algebroids.

In this note we apply the above idea to the functor of inverse
image of a Lie (respectively, Courant) algebroid under a map
manifolds. The construction of inverse image for Courant
algebroids has appeared in the literature, at least in special
cases; see \citep{LM,SV,V}.  We study the
naturality properties of the inverse image functor as well as its
behavior with respect to some standard constructions. We formulate
and prove descent for Lie (respectively, Courant) algebroids along
a surjective submersion.

In addition, we take the opportunity to relate the inverse image for Courant algebroids to the notion of Dirac structure with support and of Courant morphisms of \citep{AX} (see also \citep{BIS}).

The paper is organized as follows. In Section \ref{section: DG manifolds} we briefly review the requisite notions from the theory of DG manifolds. In Section \ref{section: Transgression for Courant algebroids} we recall the ``odd transgression" for Courant algebroids of \citep{BR} and further develop its properties. In Section \ref{section: The inverse image functor} we construct the inverse image functor for $\mathcal{O}$-modules equipped with an ``anchor" map to the tangent bundle and study its behavior in compositions of maps. Section \ref{section: Localization and descent} is devoted to setting up the descent problem. In Section \ref{section: Inverse image for Lie algebroids} we apply previously obtained results to the setting of DG Lie algebroids. In Section \ref{section: Inverse image for Courant algebroids} we construct the inverse image functor for Courant algebroids and study its relationship to various notions of Courant algebroid theory.

\subsection{Notation}
In order to simplify notations in numerous signs we will write
``$a$" instead of ``$\deg(a)$" in expressions appearing in
exponents of $-1$. For example, $(-1)^{ab - 1}$ stands for
$(-1)^{\deg(a) \cdot \deg(b) - 1}$.

Throughout the paper ``manifold" means a $C^\infty$, real analytic or complex manifold.
For a manifold $X$ we denote by $\mathcal{O}_X$ the corresponding structure sheaf of \emph{real or complex valued} $C^\infty$, respectively analytic
or holomorphic functions. We denote by $\mathcal{T}_X$ (respectively, by $\Omega_X^k$) the sheaf of real or complex valued vector fields
(respectively, differential forms of degree $k$) on $X$.

For a sheaf of algebras $R$ on $X$, let $\Mod{R}$ denotes the (abelian) category of sheaves of $R$-modules on $X$.

For a map of manifolds $f \colon Y \to X$ and $\mathcal{E}\in\Mod{\mathcal{O}_X}$, we denote by
\[
f^*\mathcal{E} := \mathcal{O}_Y\otimes_{f^{-1}\mathcal{O}_X}f^{-1}\mathcal{E}
\]
the induced $\mathcal{O}_Y$-module, where $f^{-1}$ denotes the sheaf-theoretic inverse image. The assignment $\mathcal{E} \mapsto f^*\mathcal{E}$ extends to a functor
\[
f^* \colon \Mod{\mathcal{O}_X} \to \Mod{\mathcal{O}_Y}.
\]

For a sheaf $\mathcal{F}$ on $X$, by $a\in\mathcal{F}$ we mean that $a$ is a local section of $\mathcal{F}$, i.e. there is an open set $U$ of $X$ such that $a\in\Gamma(U;\mathcal{F})$.

For a $\mathbb{Z}$-graded object $A$ and $i\in\mathbb{Z}$ we denote by $A^i$ the graded component of $A$ of degree $i$.

\section{DG manifolds}\label{section: DG manifolds}
In what follows we use notation introduced in \citep{BR}.

\subsection{The category of DG manifolds}
For the purposes of the present note a \emph{differential-graded manifold (DG-manifold)} is a pair
$\mathfrak{X} := (X, \mathcal{O}_\mathfrak{X})$, where $X$ is a manifold and $\mathcal{O}_\mathfrak{X}$
is a sheaf of commutative differential-graded algebras (CDGA) on $X$ locally isomorphic to one of the form
$\mathcal{O}_X\otimes\Symm(E)$, where $\Symm(E)$ is the symmetric algebra of a finite-dimensional graded vector space $E$.

Let $\mathfrak{X} = (X, \mathcal{O}_\mathfrak{X})$ and $\mathfrak{Y} = (Y, \mathcal{O}_\mathfrak{Y})$ be DG-manifolds. A morphism $\phi\colon \mathfrak{X} \to \mathfrak{Y}$ is a morphism of ringed spaces, which is to say a map $\phi\colon X \to Y$ of manifolds together with a morphism of differential-graded algebras
$\phi^*\colon \phi^{-1}\mathcal{O}_\mathfrak{Y} \to \mathcal{O}_\mathfrak{X}$ compatible with the canoncical map
$\phi^{-1}\mathcal{O}_Y \to \mathcal{O}_X$.

We denote the category of DG-manifolds by $\DGMan$. Let $\DGMan^+$
denote the full subcategory of DG-manifolds $\mathfrak{X} = (X, \mathcal{O}_\mathfrak{X})$ such that $\mathcal{O}_\mathfrak{X}^i = 0$ if $i < N$ for some $N\in\mathbb{Z}$.

For $\mathfrak{X} \in \DGMan$ we denote by $\Mod{\mathcal{O}_\mathfrak{X}}$ the category of sheaves of differential-graded modules over the structure sheaf $\mathcal{O}_\mathfrak{X}$.

\begin{example}\label{example: shifted T}
An ordinary manifold is an example of a DG-manifold with the
structure sheaf concentrated in degree zero. Each ordinary
manifold $X$ determines a DG-manifold $X^\sharp \in \DGMan^+$
defined by $X^\sharp = (X,\Omega^\bullet_X, d)$ and frequently denoted by $T[1]X$ in the literature. There is a canonical morphism $X \to X^\sharp$ of DG-manifolds defined by the
canonical map $\Omega^\bullet_X \to \mathcal{O}_X$.
\end{example}
\begin{example}\label{example: odd line}
Let $\vec{\mathfrak{t}}$ denote the DG-manifold with the
underlying space  consisting of one point and the DG-algebra of
functions $\mathcal{O}_{\vec{\mathfrak{t}}} =
\mathbb{C}[\epsilon]$, the free graded commutative algebra with one
generator $\epsilon$ of degree $-1$ and the differential
$\partial_\epsilon \colon \epsilon \mapsto 1$. Note that
$\vec{\mathfrak{t}} \in \DGMan^+$.
\end{example}

The category $\DGMan^+$ has finite products (\citep{BR}, Lemma 2.1). For
$\mathfrak{X} = (X, \mathcal{O}_\mathfrak{X}), \mathfrak{Y} = (Y,
\mathcal{O}_\mathfrak{Y}) \in \DGMan^+$ the product is given by $(X\times Y,
\mathcal{O}_{\mathfrak{X}\times\mathfrak{Y}})$ with
\[
\mathcal{O}_{\mathfrak{X}\times\mathfrak{Y}} := \mathcal{O}_{X\times Y}\otimes_{\pr_X^{-1}\mathcal{O}_X\otimes\pr_Y^{-1}\mathcal{O}_Y} \pr_X^{-1}\mathcal{O}_\mathfrak{X}\otimes\pr_Y^{-1}\mathcal{O}_\mathfrak{Y} .
\]


\subsection{The odd path space}
For a manifold $X$ the mapping space $X^{\vec{\mathfrak{t}}}$ is represented by the DG manifold $X^\sharp$ of Example \ref{example: shifted T} (\citep{BR}, Theorem 2.1). The evaluation map
\[
\ev \colon X^\sharp \times \vec{\mathfrak{t}} \to X
\]
corresponds to the morphism of ``pull-back of functions''
\[
\ev^* \colon \mathcal{O}_X \to \mathcal{O}_{X^\sharp}[\epsilon] := \mathcal{O}_{X^\sharp\times \vec{\mathfrak{t}}}
\]
given by $f \mapsto f + df\cdot\epsilon$.

There is a short exact sequence of \emph{graded} $\mathcal{O}_{X^\sharp}$-modules
\[
0 \to \mathcal{O}_{X^\sharp}\otimes_{\mathcal{O}_X} \mathcal{T}_X[1] \to \mathcal{T}_{X^\sharp} \to \mathcal{O}_{X^\sharp}\otimes_{\mathcal{O}_X} \mathcal{T}_X \to 0 ,
\]
where $\mathcal{O}_{X^\sharp}\otimes_{\mathcal{O}_X} \mathcal{T}_X[1] \cong \mathcal{T}_{X^\sharp/X}$.

\subsection{Immersions and submersions}
\begin{proposition}
If a map of manifolds $f\colon Y \to X$ is an immersion (respectively, submersion), then the induced map $f^\sharp\colon Y^\sharp \to X^\sharp$ is an immersion (respectively, submersion).
\end{proposition}
\begin{proof}
The derivative $df^\sharp\colon \mathcal{T}_{Y^\sharp} \to f^{\sharp *}\mathcal{T}_{X^\sharp}$ gives rise to the map of short exact sequences
\[
\begin{CD}
0 @>>> \mathcal{O}_{Y^\sharp}\otimes_{\mathcal{O}_Y} \mathcal{T}_Y[1] @>>> \mathcal{T}_{Y^\sharp} @>>> \mathcal{O}_{Y^\sharp}\otimes_{\mathcal{O}_Y} \mathcal{T}_Y @>>> 0 \\
& & @V{\id\otimes df}VV @V{df^\sharp}VV @VV{\id\otimes df}V \\
0 @>>> \mathcal{O}_{Y^\sharp}\otimes_{f^{-1}\mathcal{O}_X} \mathcal{T}_X[1] @>>> \mathcal{T}_{X^\sharp} @>>> \mathcal{O}_{Y^\sharp}\otimes_{f^{-1}\mathcal{O}_X} \mathcal{T}_X @>>> 0
\end{CD}
\]
and the claim follows from the five-lemma.
\end{proof}

\subsection{Transgression for $\mathcal{O}$-modules}
We denote by $\pr \colon X^\sharp \times \vec{\mathfrak{t}} \to X^\sharp$ the canonical projection. The diagram
\[
\begin{CD}
X^\sharp\times\vec{\mathfrak{t}} @>{\ev}>> X \\
@V{\pr}VV \\
X^\sharp
\end{CD}
\]
gives rise to the functor
\begin{equation}\label{pr-ev for O-mod}
\pr_*\ev^* \colon \Mod{\mathcal{O}_X} \to \Mod{\mathcal{O}_{X^\sharp}}.
\end{equation}
Since the underlying space of both $X^\sharp$ and $X^\sharp\times\vec{\mathfrak{t}}$ is equal to $X$, the functor $\ev^* \colon \Mod{\mathcal{O}_X} \to \Mod{\mathcal{O}_{X^\sharp\times\vec{\mathfrak{t}}}}$ is given by $\ev^*\mathcal{E} = \mathcal{O}_{X^\sharp}[\epsilon]\otimes_{\mathcal{O}_X}\mathcal{E}$ and the effect of the functor $\pr_*$ amounts to restriction of scalars along the unit map $\mathcal{O}_{X^\sharp} \to \mathcal{O}_{X^\sharp}[\epsilon]$.

\subsection{Lie algebroids}\label{subsection: Lie algebroids}
An $\mathcal{O}_X$-Lie algebroid structure on an $\mathcal{O}_X$-module $\mathcal{A}$ consists of
\begin{enumerate}
\item a structure of a sheaf of $\mathbb{C}$-Lie algebras $[\ ,\ ] \colon \mathcal{A}\otimes_\mathbb{C}\mathcal{A} \to \mathcal{A}$;

\item an $\mathcal{O}_X$-linear map $\sigma \colon \mathcal{A} \to
\mathcal{T}_X$ of Lie algebras called the \emph{anchor map}.
\end{enumerate}
These data are required to satisfy the compatibility condition (Leibniz rule)
\[
[a,f\cdot b] = \sigma(a)(f)\cdot b + f\cdot [a,b]
\]
for $a,b \in \mathcal{A}$ and $f \in \mathcal{O}_X$.

A morphism of $\mathcal{O}_X$-Lie algebroids $\phi \colon
\mathcal{A}_1 \to \mathcal{A}_2$ is an $\mathcal{O}_X$-linear map
of Lie algebras which commutes with respective anchor maps.

With the above definition of morphisms $\mathcal{O}_X$-Lie algebroids form a category denoted $\LA{\mathcal{O}_X}$.

The notion of Lie algebroid generalizes readily to the DG context.

\subsection{Transgression for Lie algebroids}\label{subsection: Transgression for Lie algebroids}
Suppose that $\mathcal{A}$ is an $\mathcal{O}_X$-Lie algebroid as in \ref{subsection: Lie algebroids}. It is shown in \citep{BR}, 3.5, that the $\mathcal{O}_{X^\sharp}$-module $\pr_*\ev^*\mathcal{A}$ admits a canonical structure of a $\mathcal{O}_{X^\sharp}$-Lie algebroid, denoted henceforth by $\mathcal{A}^\sharp$. Moreover, the assignment $\mathcal{A} \mapsto \mathcal{A}^\sharp$ extends to a functor
\[
\left(\ \right)^\sharp \colon \LA{\mathcal{O}_X} \to \LA{\mathcal{O}_{X^\sharp}}
\]
which preserves terminal objects, i.e. the canonical map $\mathcal{T}_X^\sharp \to \mathcal{T}_{X^\sharp}$ is an isomorphism, and products.

\section{Transgression for Courant algebroids}\label{section: Transgression for Courant algebroids}
\subsection{Marked Lie algebroids}
Suppose that $\mathfrak{X} = (X, \mathcal{O}_\mathfrak{X})$ is a DG-manifold.

A \emph{marked $\mathcal{O}_\mathfrak{X}$-Lie algebroid} is a pair
$(\mathcal{A}, \mathfrak{c})$, where $\mathcal{A}$ is a
$\mathcal{O}_\mathfrak{X}$-Lie algebroid and
$\mathfrak{c}\in\Gamma(X;\mathcal{A})$ is a homogeneous central
section (i.e. $[\mathfrak{c},\mathcal{A}] = 0$). The section $\mathfrak{c}$ is called a marking.

It is easy to see (\citep{BR}, Lemma 3.6) that any (homogeneous) central section belongs to the kernel of the anchor map.

A morphism $\phi\colon (\mathcal{A}_1, \mathfrak{c}_1) \to
(\mathcal{A}_2, \mathfrak{c}_2)$ is a morphism of Lie algebroids
$\phi\colon \mathcal{A}_1 \to \mathcal{A}_2$ such that
$\phi(\mathfrak{c}_1) = \mathfrak{c}_2$. In particular,
$\mathfrak{c}_1$ and $\mathfrak{c}_2$ have the same degree.

With the above definitions marked $\mathcal{O}_\mathfrak{X}$-Lie
algebroids and morphisms thereof form a category denoted
$\LA{\mathcal{O}_\mathfrak{X}}^\star$. The full subcategory of
marked $\mathcal{O}_\mathfrak{X}$-Lie algebroids $(\mathcal{A},
\mathfrak{c})$ with $\deg\mathfrak{c} = n$ is denoted
$\LA{\mathcal{O}_\mathfrak{X}}^\star_n$.

For a marked Lie algebroid $(\mathcal{A}, \mathfrak{c})$ with $\deg\mathfrak{c} = n$ the Lie algebroid structure on $\mathcal{A}$ descends to
\[
\overline{\mathcal{A}} := \coker(\mathcal{O}_\mathfrak{X}[n] \xrightarrow{\cdot\mathfrak{c}} \mathcal{A})
\]
(\citep{BR}, Lemma 3.7). The assignment $(\mathcal{A}, \mathfrak{c}) \mapsto \overline{\mathcal{A}}$ extends to a functor
\[
\overline{(\ )} \colon \LA{\mathcal{O}_\mathfrak{X}}^\star \to \LA{\mathcal{O}_\mathfrak{X}}.
\]

\begin{example}\label{trivial-marked-Lie-alg}
The structure sheaf $\mathcal{O}_\mathfrak{X}[n]$ has a canonical
structure of a marked $\mathcal{O}_\mathfrak{X}$-Lie algebroid,
whose non trivial part is the marking $id \colon
\mathcal{O}_\mathfrak{X}[n] \to \mathcal{O}_\mathfrak{X}[n]$. In
particular, the marking is a morphism of marked Lie algebroids and the $\mathcal{O}_\mathfrak{X}$-Lie algebroid
$\overline{\mathcal{O}_\mathfrak{X}[n]}$ is zero.
\end{example}

\subsection{$\mathcal{O}_\mathfrak{X}[n]$-extensions}
Suppose that $\mathcal{B}$ is a $\mathcal{O}_\mathfrak{X}$-Lie algebroid.

An $\mathcal{O}_\mathfrak{X}[n]$-extension of $\mathcal{B}$ is a
marked $\mathcal{O}_\mathfrak{X}$-Lie algebroid $(\mathcal{A},
\mathfrak{c})$ with $\deg\mathfrak{c} = n$ \emph{together with} a morphism $\mathcal{A} \to \mathcal{B}$ such that the sequence
\[
0 \to \mathcal{O}_\mathfrak{X}[n] \xrightarrow{\cdot\mathfrak{c}} \mathcal{A} \to \mathcal{B} \to 0
\]
is exact. A morphism of $\mathcal{O}_\mathfrak{X}[n]$-extensions of $\mathcal{B}$ is a morphism of marked Lie algebroids which induces the identity map on $\mathcal{B}$. Such a map is necessarily an isomorphism. We denote the category (groupoid) of $\mathcal{O}_\mathfrak{X}[n]$-extensions of $\mathcal{B}$ by $\OExt{n}(\mathcal{B})$.

The category $\OExt{n}(\mathcal{B})$ has a canonical structure of a `$\mathbb{C}$-vector space in categories' (hence, in particular, that of a Picard groupoid). Namely, given extensions $\mathcal{A}_1,\ldots,\mathcal{A}_m$ and complex numbers $\lambda_1,\ldots,\lambda_m$ the `linear combination' $\lambda_1\mathcal{A}\dotplus\cdots\dotplus\lambda_m\mathcal{A}_m$ is defined by the push-out diagram
\[
\begin{CD}
\mathcal{O}_\mathfrak{X}[n]\times\cdots\times\mathcal{O}_\mathfrak{X}[n] @>>> \mathcal{A}_1\times_\mathcal{B}\cdots\times_\mathcal{B}\mathcal{A}_m \\
@VVV @VVV \\
\mathcal{O}_\mathfrak{X}[n] @>>> \lambda_1\mathcal{A}_1\dotplus\cdots\dotplus\lambda_m\mathcal{A}_m
\end{CD}
\]
where the left vertical arrow is given by $(\alpha_1,\ldots,\alpha_m) \mapsto \sum\limits_i\lambda_i\alpha_i$. The bracket on $\lambda_1\mathcal{A}\dotplus\cdots\dotplus\lambda_m\mathcal{A}_m$ is characterized by the fact that the right vertical map is a morphism of Lie algebras.


\subsection{Courant algebroids}
Courant algebroids were introduced in \citep{LWX}, \citep{R} and \citep{BCh}. For comparison of the following definition with the one encountered in the literature see Remark \ref{remark: comparison of def}.

A Courant algebroid is an $\mathcal{O}_X$-module $\mathcal{Q}$ equipped with
\begin{enumerate}
\item a structure of a Leibniz $\mathbb{C}$-algebra
\[
\{\ ,\ \} \colon \mathcal{Q}\otimes_\mathbb{C}\mathcal{Q} \to \mathcal{Q} ;
\]

\item an $\mathcal{O}_X$-linear map of Leibniz algebras (the anchor map)
\[
\pi \colon \mathcal{Q} \to \mathcal{T}_X ;
\]

\item a symmetric $\mathcal{O}_X$-bilinear pairing\footnote{not assumed to be non-degenerate}
\[
\ip\colon \mathcal{Q}\otimes_{\mathcal{O}_X}\mathcal{Q} \to \mathcal{O}_X;
\]

\item an $\mathcal{O}_X$-linear map (the co-anchor map)
\[
\pi^\dagger \colon \Omega_X^1 \to \mathcal{Q} .
\]
\end{enumerate}
These data are required to satisfy
\begin{eqnarray}
\pi\circ\pi^\dagger & = & 0 \label{complex}\\
\{q_1, fq_2\} & = & f\{q_1,q_2\} + \pi(q_1)(f)q_2 \label{leibniz}\\
\langle \{q,q_1\}, q_2 \rangle + \langle q_1, \{q, q_2\}\rangle & = & L_{\pi(q)}\langle q_1,q_2 \rangle \label{ip invariance}\\
\{q, \pi^\dagger(\alpha)\} & = & \pi^\dagger(L_{\pi(q)}(\alpha)) \label{forms left ideal}\\
\langle q, \pi^\dagger(\alpha) \rangle & = & \iota_{\pi(q)}\alpha \label{adjunction}\\
\{q_1, q_2\} + \{q_2, q_1\} & = & \pi^\dagger(d\langle q_1, q_2\rangle) \label{symmetrizer}
\end{eqnarray}
for $f\in \mathcal{O}_X$ and $q, q_1, q_2 \in \mathcal{Q}$.

A morphism $\phi \colon \mathcal{Q}_1 \to \mathcal{Q}_2$ of Courant algebroids on $X$ is an
$\mathcal{O}_X$-linear map of Leibniz $\mathbb{C}$-algebras such
that the diagram
\[
\begin{CD}
\Omega^1_X @>{\pi^\dagger_1}>> \mathcal{Q}_1 @>{\pi_1}>> \mathcal{T}_X \\
@| @V{\phi}VV @| \\
\Omega^1_X @>{\pi^\dagger_2}>> \mathcal{Q}_2 @>{\pi_2}>> \mathcal{T}_X
\end{CD}
\]
is commutative.

With the above definitions Courant algebroids on $X$ and morphisms thereof form a category henceforth denoted $\CA(X)$.

\begin{remark}\label{remark: comparison of def}
The orginal definition of Courant algebroid is due to \citep{LWX}, based on \citep{C} and \citep{D}, where the additional assumption of non-degeneracy of the pairing is made. Under the latter assumption the co-anchor map is uniquely determined by the anchor map and, hence, does not appear as a separate item. The broader concept of Courant algebroid as described by the definition was advanced in \citep{B} and permits, for example, regarding the kernel of the anchor map as a Courant algebroid with trivial anchor.
\end{remark}

Suppose that $\mathcal{Q}$ is a Courant algebroid on $X$. Let
\[
\overline{\mathcal{Q}} = \coker(\pi^\dagger) .
\]
The Courant algebroid structure on $\mathcal{Q}$ descends to a structure on a Lie algebroid on $\overline{\mathcal{Q}}$. In what follows we refer to the Lie algebroid $\overline{\mathcal{Q}}$ as \emph{the Lie algebroid associated to the Courant algebroid $\mathcal{Q}$}.

The assignment $\mathcal{Q} \mapsto \overline{\mathcal{Q}}$ extends to a functor
\[
\overline{(\ )} \colon \CA(X) \longrightarrow \LA{\mathcal{O}_X} .
\]

For a Courant algebroid $\mathcal{Q}$ the \emph{opposite} Courant algebroid, denoted $\mathcal{Q}^\op$, has the same underlying $\mathcal{O}_X$-module as $\mathcal{Q}$, the same Leibniz bracket (i.e. $\{\ ,\ \}^\op = \{\ ,\ \})$, same anchor map $\pi^\op = \pi$, the symmetric pairing $\ip^\op = - \ip$ and the co-anchor map $\pi^{\op\dagger} = - \pi^\dagger$.

Suppose that $\mathcal{Q}_1$ and $\mathcal{Q}_2$ are Courant algebroids. The Courant algebroid $\mathcal{Q}_1\dotplus\mathcal{Q}_2$ is defined by the push-out square
\[
\begin{CD}
\Omega^1_X\times\Omega^1_X @>>> \mathcal{Q}_1\times_{\mathcal{T}_X}\mathcal{Q}_2 \\
@V{+}VV @VVV \\
\Omega^1_X @>>> \mathcal{Q}_1\dotplus\mathcal{Q}_2
\end{CD}
\]
and equipped with the component-wise operations.

\subsection{Courant extensions}
Suppose that $\mathcal{A}$ is a $\mathcal{O}_X$-Lie algebroid. A Courant extension of $\mathcal{A}$ is a Courant algebroid $\mathcal{Q}$ together with the identification $\overline{\mathcal{Q}} \cong \mathcal{A}$ such that the sequence
\begin{equation}\label{CExt transitive exact sequence}
0 \to \Omega^1_X \xrightarrow{\pi^\dagger} \mathcal{Q} \to \mathcal{A} \to 0
\end{equation}
is exact.

A morphism $\phi\colon \mathcal{Q}_1 \to \mathcal{Q}_2$ of Courant extensions of $\mathcal{A}$ is a morphism of Courant algebroids which is compatible with the identifications $\overline{\mathcal{Q}_i} \cong \mathcal{A}$. We denote the category of Courant extension of $\mathcal{A}$ by $\CExt{\mathcal{A}}$. A morphism in $\CExt{\mathcal{A}}$ induces a morphism of
associated short exact sequences \eqref{CExt transitive exact
sequence}, hence is an isomorphism of  underlying
$\mathcal{O}_X$-modules. It is easy to see that the inverse map
is, in fact, a morphism of Courant algebroids. Consequently,
$\CExt{\mathcal{A}}$ is a groupoid.

\subsection{Linear algebra}
The category $\CExt{\mathcal{A}}$ has a canonical structure of a `$\mathbb{C}$-vector space in categories' (hence, in particular, that of a Picard groupoid). Namely, given extensions $\mathcal{Q}_1,\ldots,\mathcal{Q}_n$ and complex numbers $\lambda_1,\ldots,\lambda_n$ the `linear combination' $\lambda_1\mathcal{Q}\dotplus\cdots\dotplus\lambda_n\mathcal{Q}_n$ is defined by the push-out diagram
\[
\begin{CD}
\Omega^1_X\times\cdots\times\Omega^1_X @>>> \mathcal{Q}_1\times_\mathcal{A}\cdots\times_\mathcal{A}\mathcal{Q}_n \\
@VVV @VVV \\
\Omega^1_X @>>> \lambda_1\mathcal{Q}\dotplus\cdots\dotplus\lambda_n\mathcal{Q}_n
\end{CD}
\]
where the left vertical arrow is given by $(\alpha_1,\ldots,\alpha_n) \mapsto \sum\limits_i\lambda_i\alpha_i$. The Leibniz bracket on $\lambda_1\mathcal{Q}\dotplus\cdots\dotplus\lambda_n\mathcal{Q}_n$ is characterized by the fact that the right vertical map is a morphism of Leibniz algebras.

\subsection{Exact Courant algebroids}
A Courant algebroid $\mathcal{Q}$ is called \emph{exact} if the map $\overline{\mathcal{Q}} \to \mathcal{T}_X$ is an isomorphism. Thus, an exact Courant algebroid is a Courant extension of $\mathcal{T}_X$. We denote the category of exact Courant algebroids by $\ECA(X) := \CExt{\mathcal{T}_X}$.

\subsection{Transitive Courant algebroids}
A Courant algebroid is called \emph{transitive} if the associated Lie algebroid is, which is to say, the anchor map is an epimorphism. If $\mathcal{Q}$ is a transitive Courant algebroid, the sequence
\[
0 \to \Omega^1_X \to \mathcal{Q} \to \overline{\mathcal{Q}} \to 0
\]
is exact.

\subsection{Transgression for Courant algebroids}\label{subsection: Transgression for Courant algebroids}
We denote by
\begin{equation}\label{integration map}
{\textstyle\int}\colon \pr_*\ev^*\Omega_X^1 = \mathcal{O}_{X^\sharp}[\epsilon]\otimes_{\mathcal{O}_X}\Omega_X^1 \to \mathcal{O}_{X^\sharp}[2] .
\end{equation}
the map of $\mathcal{O}_{X^\sharp}$-modules whose component of degree $-1$ is the identity map.

For a Courant algebroid $\mathcal{Q}$ the marked Lie algebroid $(\tau\mathcal{Q},\mathfrak{c}) \in \LA{\mathcal{O}_\mathfrak{X}}^\star_2$ is given by the $\mathcal{O}_{X^\sharp}$-module
\[
\tau\mathcal{Q} := \coker(\pr_*\ev^*\Omega_X^1
\xrightarrow{(\int,-\pr_*\ev^*(\pi^\dagger))}
\mathcal{O}_{X^\sharp}[2]\oplus \pr_*\ev^*\mathcal{Q}) \ ,
\]
where $\int$ is the map \eqref{integration map}. In other words, the square
\begin{equation}\label{defn tauQ coCart square}
\begin{CD}
\pr_*\ev^*\Omega_X^1 @>{\pr_*\ev^*(\pi^\dagger)}>> \pr_*\ev^*\mathcal{Q} \\
@V{\int}VV @VVV \\
\mathcal{O}_{X^\sharp}[2] @>>> \tau\mathcal{Q}
\end{CD}
\end{equation}
is cocartesian.

The anchor map is induced by $\pr_*\ev^*(\pi)$ and the marking $\mathfrak{c}\in\Gamma(X;\tau\mathcal{Q}^{-2})$ is the image of $1\in\Gamma(X;(\mathcal{O}_{X^\sharp}[2])^{-2})$ under the bottom horizontal map in \eqref{defn tauQ coCart square}.

The bracket is the extension by Leibniz rule of
\begin{enumerate}
\item $[\mathfrak{c},\tau\mathcal{Q}] = 0$

\item $[q\cdot\epsilon, \beta] = -(-1)^{|\beta|}[\beta,q\cdot\epsilon] = \iota_{\pi(q)}\beta$

\item $[q,\beta] = -[\beta, q] = L_{\pi(q)}\beta$

\item
$[q_1\cdot\epsilon,q_2\cdot\epsilon]^{-1,-1} = \langle
q_1,q_2\rangle \in (\mathcal{O}_{X^\sharp}[2])^{-2}=\mathcal{O}_X$

\item
$[q_1,q_2\cdot\epsilon]^{0,-1} = \{q_1,q_2\}\cdot\epsilon$

\item $[q_1\cdot\epsilon, q_2]^{-1,0} = -d\langle
q_1,q_2\rangle + \{q_1,q_2\}\cdot\epsilon$
\end{enumerate}

The marked $\mathcal{O}_{X^\sharp}$-Lie algebroid $\tau\mathcal{Q}$ enjoys the following properties:
\begin{enumerate}
\item The natural map $\overline{\mathcal{Q}}^\sharp \to \overline{\tau\mathcal{Q}}$ is an isomorphism.

\item The canonical map $\mathcal{Q} \to \tau\mathcal{Q}^{-1}$ is an isomorphism.

\item $\mathcal{Q}$ is a Courant extension of $\mathcal{A}$ if and only if $\tau\mathcal{Q}$ is a $\mathcal{O}_{X^\sharp}[2]$-extension of $\mathcal{A}^\sharp$.
\end{enumerate}

The assignment $\mathcal{Q} \mapsto \tau\mathcal{Q}$ extends to a functor
\[
\tau \colon \CA(X) \longrightarrow \LA{\mathcal{O}_{X^\sharp}}^\star_2 .
\]

Suppose that $(\mathcal{B},\mathfrak{c}) \in \LA{\mathcal{O}_{X^\sharp}}^\star_2$ satisfies $\overline{\mathcal{B}}^i = 0$ for $i \leqslant -2$. Then, the derived bracket and the $(-1,-1)$-component of the Lie bracket together with the components of degree -1 of the anchor map and of $\mathcal{O}_{X^\sharp}[2] \xrightarrow{\cdot\mathfrak{c}} \mathcal{B}$ endow $\mathcal{B}^{-1}$ with a structure of a Courant algebroid (see \citep{BR}, Lemma 6.1). We denote this Courant algebroid structure on $\mathcal{B}^{-1}$  by $\Cour(\mathcal{B},\mathfrak{c})$.

For any Courant algebroid $\mathcal{Q}$ the marked Lie algebroid $\tau\mathcal{Q}$ satisfies the above requirements. Moreover, $\Cour(\tau\mathcal{Q}) = \mathcal{Q}$.

\begin{proposition}\label{lemma: OExt and CExt}
Suppose that $\mathcal{A}$ is a $\mathcal{O}_X$-Lie algebroid $\mathcal{A}$. The functor $\tau$ restricts to an equivalence of categories
\[
\tau \colon \CExt{\mathcal{A}}  \longrightarrow \OExt{2}(\mathcal{\mathcal{A}^\sharp})
\]
whose quasi-inverse is given by the functor
\[
\Cour \colon \OExt{2}(\mathcal{A}^\sharp) \to \CExt{\mathcal{A}} .
\]
\end{proposition}

\subsection{Transgression and linear algebra}
Suppose that $\mathcal{A}$ is a $\mathcal{O}_X$-Lie algebroid locally free of finite rank over $\mathcal{O}_X$. Then, so is any Courant extension of $\mathcal{A}$. Suppose that $\mathcal{Q}_1,\ldots,\mathcal{Q}_n \in\CExt{\mathcal{A}}$, $\lambda_1,\ldots,\lambda_n \in \mathbb{C}$. For $\alpha_1,\ldots\alpha_n \in\Omega^1_X$ let $\ell(\alpha_1,\ldots\alpha_n) = \sum\limits_i \lambda_i\alpha_i$.

The diagram
\[
\begin{CD}
\pr_*\ev^*\left(\Omega^1_X\times\cdots\times\Omega^1_X\right) @>>> \pr_*\ev^*\Omega^1_X\times\cdots\times\pr_*\ev^*\Omega^1_X \\
@VVV @VVV \\
\pr_*\ev^*\left(\mathcal{Q}_1\times_\mathcal{A}\cdots\times_\mathcal{A}\mathcal{Q}_n\right) @>>> \pr_*\ev^*\mathcal{Q}_1\times_{\mathcal{A}^\sharp}\cdots\times_{\mathcal{A}^\sharp}\pr_*\ev^*\mathcal{Q}_n
\end{CD}
\]
is commutative with horizontal maps isomorphisms. Both squares in the diagram
\[
\begin{CD}
\pr_*\ev^*\left(\Omega^1_X\times\cdots\times\Omega^1_X\right) @>>> \pr_*\ev^*\left(\mathcal{Q}_1\times_\mathcal{A}\cdots\times_\mathcal{A}\mathcal{Q}_n\right) \\
@V{\pr_*\ev^*(\ell)}VV @VVV \\
\pr_*\ev^*\Omega^1_X @>>> \pr_*\ev^*\left(\opdp_{i=1}^n\lambda_i\mathcal{Q}_i\right) \\
@V{\int}VV @VVV \\
\mathcal{O}_{X^\sharp}[2] @>>> \tau\left(\opdp_{i=1}^n\lambda_i\mathcal{Q}_i\right)
\end{CD}
\]
are co-Cartesian, hence so is the outer one and the same holds for the diagram
\[
\begin{CD}
\pr_*\ev^*\Omega^1_X\times\cdots\times\pr_*\ev^*\Omega^1_X @>>> \pr_*\ev^*\mathcal{Q}_1\times_{\mathcal{A}^\sharp}\cdots\times_{\mathcal{A}^\sharp}\pr_*\ev^*\mathcal{Q}_n \\
@V{\int\times\cdots\times\int}VV @VVV \\
\mathcal{O}_{X^\sharp}[2]\times\cdots\times\mathcal{O}_{X^\sharp}[2] @>>> \tau\mathcal{Q}_1\times_{\mathcal{A}^\sharp}\cdots\times_{\mathcal{A}^\sharp}\tau\mathcal{Q}_n \\
@V{\ell}VV @VVV \\
\mathcal{O}_{X^\sharp}[2] @>>> \opdp_{i=1}^n\lambda_i\tau\mathcal{Q}_i
\end{CD}
\]
Since the diagram
\[
\begin{CD}
\pr_*\ev^*\left(\Omega^1_X\times\cdots\times\Omega^1_X\right) @>>> \pr_*\ev^*\Omega^1_X  @>>> \mathcal{O}_{X^\sharp}[2] \\
@VVV @VVV @VVV \\
\pr_*\ev^*\Omega^1_X\times\cdots\times\pr_*\ev^*\Omega^1_X @>>> \mathcal{O}_{X^\sharp}[2]\times\cdots\times\mathcal{O}_{X^\sharp}[2] @>>> \mathcal{O}_{X^\sharp}[2]
\end{CD}
\]
is commutative with vertical maps isomorphisms it follows that there is a canonical isomorphism
\[
\tau\left(\opdp_{i=1}^n\lambda_i\mathcal{Q}_i\right) \cong \opdp_{i=1}^n\lambda_i\tau\mathcal{Q}_i .
\]
We leave the details of the proof of the following lemma to the reader.
\begin{proposition}\label{lemma: tau and Q are linear}
The functors
\[
\tau \colon \CExt{\mathcal{A}}  \longrightarrow \OExt{2}(\mathcal{\mathcal{A}^\sharp})\ \colon\! \mathtt{Q}
\]
are morphisms of $\mathbb{C}$-vector spaces in categories (and, in particular, of Picard groupoids).
\end{proposition}

\section{The inverse image functor}\label{section: The inverse image functor}

\subsection{Inverse image and fiber product for $\mathcal{O}$-modules}\label{subsection: condition C}
Suppose that $\phi \colon \mathfrak{Y} \to \mathfrak{X}$ is a morphism of DG manifolds and $A \xrightarrow{s} C \xleftarrow{t} B$ are morphisms in $\Mod{\mathcal{O}_\mathfrak{X}}$.
The sequence
\[
0 \to A\times_C B \to A\times B \xrightarrow{s-t} C
\]
is exact. Assume furthermore that $A$, $B$ and $C$ are locally free of finite rank. In what follows we will consider the following condition on morphisms $A \xrightarrow{s} C \xleftarrow{t} B$:
\begin{equation}\label{condition C}\tag{C}
\text{$A\times_C B = \ker(s-t)$ and $\coker(s-t)$ are locally free}
\end{equation}
Condition \eqref{condition C} is trivially fulfilled if at least one of the two maps $s$ and $t$ is an epimorphism.

\begin{lemma}\label{lemma: condition C}
Suppose that the morphisms $A \xrightarrow{s} C \xleftarrow{t} B$ satisfy the condition \eqref{condition C}. Then, the canonical morphism $\phi^*(A\times_C B) \to \phi^*A\times_{\phi^*C}\phi^*B$ is an isomorphism and the maps $\phi^*A \xrightarrow{\phi^*(s)} \phi^*C \xleftarrow{\phi^*(t)} \phi^*B$ satisfy the condition \eqref{condition C}.
\end{lemma}
\begin{proof}
Applying the functor $\phi^*$ to the exact sequence
\[
0 \to A\times_C B \to A\times B \xrightarrow{s-t} C \to \coker(s-t) \to 0
\]
or locally free $\mathcal{O}_\mathfrak{X}$-modules we obtain the commutative diagram with exact rows
\[
\begin{CD}
0 @>>> \phi^*(A\times_C B) @>>> \phi^*(A\times B) @>>{\phi^*(s-t)}> \phi^*C @>>> \phi^*\coker(s-t) @>>> 0 \\
& & @VVV @VVV @VV{\id}V @VVV \\
0 @>>> \phi^*A\times_{\phi^*C} \phi^*B @>>> \phi^*A\times \phi^*B @>>{\phi^*s-\phi^*t}> \phi^*C @>>> \coker(\phi^*s-\phi^*t) @>>> 0
\end{CD}
\]
Since the map $\phi^*(A\times B) \to \phi^*A\times \phi^*B$ is an isomorphism it follows that so is the map $\phi^*(A\times_C B) \to \phi^*A\times_{\phi^*C} \phi^*B$.
\end{proof}

\subsection{$\mathcal{O}$-modules over $\mathcal{T}$}
Suppose that $\mathfrak{X} = (X,\mathcal{O}_\mathfrak{X})$ is a DG manifold.

We denote by $\Mod{\mathcal{O}_\mathfrak{X}}/\mathcal{T}_\mathfrak{X}$ the category of $\mathcal{O}_\mathfrak{X}$-modules over $\mathcal{T}_\mathfrak{X}$, i.e. the category of pairs $(\mathcal{E},\pi)$, where $\mathcal{E}\in \Mod{\mathcal{O}_\mathfrak{X}}$ and $\pi\colon \mathcal{E} \to \mathcal{T}_\mathfrak{X}$ is a morphism $\Mod{\mathcal{O}_\mathfrak{X}}$ in referred to as \emph{the anchor}.

A morphism $t \colon (\mathcal{E}_1,\pi_1) \to (\mathcal{E}_2,\pi_2)$ in $\Mod{\mathcal{O}_\mathfrak{X}}/\mathcal{T}_\mathfrak{X}$ is a morphism of $\mathcal{O}_\mathfrak{X}$-modules $t \colon \mathcal{E}_1 \to \mathcal{E}_2$ such that $\pi_2\circ t = \pi_1$.

The category $\Mod{\mathcal{O}_\mathfrak{X}}/\mathcal{T}_\mathfrak{X}$ has a terminal object, namely $\mathcal{T}_\mathfrak{X} := (\mathcal{T}_\mathfrak{X},\id)$. The product $(\mathcal{E}_1,\pi_1)\times(\mathcal{E}_2,\pi_2)$ is represented by $\mathcal{E}_1\times_{\mathcal{T}_\mathfrak{X}}\mathcal{E}_2 \to \mathcal{T}_\mathfrak{X}$.

We shall call $(\mathcal{E},\pi)$ \emph{transitive} if the anchor map is surjective.

\subsection{Inverse image}\label{subsection: inverse image}
Suppose that $\phi \colon \mathfrak{Y} \to \mathfrak{X}$ is a morphism of DG manifolds.

For $\mathcal{E} := (\mathcal{E},\pi) \in \Mod{\mathcal{O}_\mathfrak{X}}/\mathcal{T}_\mathfrak{X}$ the object $\phi^+\mathcal{E} \in  \Mod{\mathcal{O}_\mathfrak{Y}}/\mathcal{T}_\mathfrak{Y}$ is defined as the left vertical arrow  in the pull-back square
\begin{equation}\label{diag: inverse image over T}
\begin{CD}
\phi^+\mathcal{E} @>{\widetilde{d\phi}}>> \phi^*\mathcal{E} \\
@V{\phi^+(\pi)}VV @VV{\phi^*(\pi)}V \\
\mathcal{T}_\mathfrak{Y} @>{d\phi}>> \phi^*\mathcal{T}_\mathfrak{X}
\end{CD}
\end{equation}
which is to say $\phi^+\mathcal{E} = \mathcal{T}_\mathfrak{Y}\times_{\phi^*\mathcal{T}_\mathfrak{X}}\phi^*\mathcal{E}$.

A morphism $t \colon (\mathcal{E}_1,\pi_1) \to (\mathcal{E}_2,\pi_2)$ in $\Mod{\mathcal{O}_\mathfrak{X}}/\mathcal{T}_\mathfrak{X}$ induces the morphism $\phi^+(t) \colon \phi^+\mathcal{E}_1 \to \phi^+\mathcal{E}_2$ in $\Mod{\mathcal{O}_\mathfrak{Y}}/\mathcal{T}_\mathfrak{Y}$ and the above construction extends to a functor, called \emph{pull-back} or \emph{inverse image under $\phi$}
\[
\phi^+ \colon \Mod{\mathcal{O}_\mathfrak{X}}/\mathcal{T}_\mathfrak{X} \to \Mod{\mathcal{O}_\mathfrak{Y}}/\mathcal{T}_\mathfrak{Y} .
\]

The functor of inverse image preserves terminal objects, that is to say $\phi^+\mathcal{T}_\mathfrak{X} = \mathcal{T}_\mathfrak{Y}$.

The functor of inverse image preserves transitivity: if $\mathcal{E} := (\mathcal{E},\pi)$ is transitive then so is $\phi^+\mathcal{E}$.

By the universal property of $\phi^{*}$ there is a unique map $\phi^{*}(\mathcal{E}_1\times_{\phi^{*}\mathcal{T}_{\mathfrak{X}}}\mathcal{E}_{2}) \to \phi^{*}\mathcal{E}_1 \times_{\mathcal{T}_{\mathfrak{Y}}} \phi^{*}\mathcal{E}_2$. Since fiber products commute with products there is a unique isomorphism $\mathcal{T}_{\mathfrak{Y}}\times_{\phi^{*}\mathcal{T}_{\mathfrak{X}}} (\phi^{*}\mathcal{E}_1\times \phi^{*}\mathcal{E}_2) \to (\mathcal{T}_{\mathfrak{Y}}\times_{\phi^{*}\mathcal{T}_{\mathfrak{X}}} \phi^{*}\mathcal{E}_2) \times_{\mathcal{T}_{\mathfrak{Y}}} (\mathcal{T}_{\mathfrak{Y}}\times_{\phi^{*}\mathcal{T}_{\mathfrak{X}}} \phi^{*}\mathcal{E}_2)$.

The composition
\[
\mathcal{T}_{\mathfrak{Y}}\times_{\phi^{*}\mathcal{T}_{\mathfrak{X}}}\phi^{*}(\mathcal{E}_1\times_{\mathcal{T}_{\mathfrak{X}}}\mathcal{E}_2) \to \mathcal{T}_{\mathfrak{Y}}\times_{\phi^{*}\mathcal{T}_{\mathfrak{X}}} (\phi^{*}\mathcal{E}_1\times_{\phi^{*}\mathcal{T}_{\mathfrak{X}}} \phi^{*}\mathcal{E}_2) \to (\mathcal{T}_{\mathfrak{Y}}\times_{\phi^{*}\mathcal{T}_{\mathfrak{X}}} \phi^{*}\mathcal{E}_2) \times_{\mathcal{T}_{\mathfrak{Y}}} (\mathcal{T}_{\mathfrak{Y}}\times_{\phi^{*}\mathcal{T}_{\mathfrak{X}}} \phi^{*}\mathcal{E}_2)
\]
gives rise to the map
\begin{equation}\label{canonical-map-and-plus}
\phi^{+}(\mathcal{E}_1 \times_{\mathcal{T}_{\mathfrak{X}}}\mathcal{E}_2) \to \phi^{+}\mathcal{E}_1 \times_{\mathcal{T}_{\mathfrak{Y}}} \phi^{+}\mathcal{E}_2
\end{equation}

\begin{lemma}\label{lemma: pull-back prod}
Suppose that $(\mathcal{E}_i,\pi_i) \in \Mod{\mathcal{O}_\mathfrak{X}}/\mathcal{T}_\mathfrak{X}$, $i = 1,2$, are locally free. Assume that the maps $\mathcal{E}_1 \xrightarrow{\pi_1} \mathcal{T}_\mathfrak{X} \xleftarrow{\pi_2} \mathcal{E}_2$ satisfy condition \eqref{condition C}. Then, the canonical map \eqref{canonical-map-and-plus} is an isomorphism.
\end{lemma}
\begin{proof}
Lemma \ref{lemma: condition C} implies that the canonical morphism
$\phi^*(\mathcal{E}_1\times_{\mathcal{T}_{\mathfrak{X}}}
\mathcal{E}_2) \to
\phi^*\mathcal{E}_1\times_{\mathcal{T}_{\mathfrak{Y}}}\phi^*\mathcal{E}_2$
is an isomorphism. Then the map \eqref{canonical-map-and-plus} is an isomorphism. Moreover, the diagrams
\[
\xymatrix{ \phi^{+}(\mathcal{E}_1\times_{\mathcal{T}_{\mathfrak{X}}}\mathcal{E}_2) \ar[r]\ar[d] & \phi^{*}(\mathcal{E}_1\times_{\mathcal{T}_{\mathfrak{X}}}\mathcal{E}_2)\ar[d] \\ \phi^{+}\mathcal{E}_1\times_{\mathcal{T}_{\mathfrak{Y}}}\phi^{+}\mathcal{E}_2 \ar[r] & \phi^{*}\mathcal{E}_1\times_{\phi^{*}\mathcal{T}_{\mathfrak{X}}}\phi^{*}\mathcal{E}_2 }
\]
\[
\xymatrix{ \phi^{+}(\mathcal{E}_1\times_{\mathcal{T}_{\mathfrak{X}}}\mathcal{E}_2) \ar[rr]\ar[rd] & & \phi^{*}\mathcal{E}_1\times_{\phi^{*}\mathcal{T}_{\mathfrak{X}}}\phi^{*}\mathcal{E}_2 \ar[ld] \\ & \mathcal{T}_{\mathfrak{Y}} &  }
\]
\[
\xymatrix{ \phi^{*}(\mathcal{E}_1\times_{\mathcal{T}_{\mathfrak{X}}}\mathcal{E}_2) \ar[rr]\ar[rd] & & \phi^{*}\mathcal{E}_1\times_{\mathcal{T}_{\mathfrak{Y}}}\phi^{*}\mathcal{E}_2 \ar[ld] \\ & \phi^{*}\mathcal{T}_{\mathfrak{X}} & }
\]
are commutative and the result follows.
\end{proof}

\subsection{Inverse image and composition}\label{Inverse and composition}
Suppose given a map $\psi \colon \mathfrak{Z} \to \mathfrak{Y}$. Applying $\psi^*$ to \eqref{diag: inverse image over T} we obtain the commutative diagram
\begin{equation}\label{diag: pull-back of inverse image}
\begin{CD}
\psi^*\phi^+\mathcal{E} @>{\psi^*(\widetilde{d\phi})}>> (\phi\circ\psi)^*\mathcal{E} \\
@V{\psi^{*}\psi^+(\pi)}VV @VV{(\phi\circ\psi)^*(\pi)}V \\
\psi^*\mathcal{T}_\mathfrak{Y} @>{\psi^*(d\phi)}>> (\phi\circ\psi)^*\mathcal{T}_\mathfrak{X}
\end{CD}
\end{equation}
which is not Cartesian in general. Combining the definition of $\psi^+$ with \eqref{diag: pull-back of inverse image} we obtain the commutative diagram
\begin{equation}\label{diag: inverse image of inverse image}
\begin{CD}
\psi^+\phi^+\mathcal{E} @>{\widetilde{d\psi}}>> \psi^*\phi^+\mathcal{E} @>{\psi^*(\widetilde{d\phi})}>> (\phi\circ\psi)^*\mathcal{E} \\
@V{\psi^+\phi^+(\pi)}VV @V{\psi^*\phi^+(\pi)}VV @VV{\psi^*\phi^*(\pi)}V \\
\mathcal{T}_\mathfrak{Z} @>{d\psi}>> \psi^*\mathcal{T}_\mathfrak{Y} @>{\psi^*(d\phi)}>> (\phi\circ\psi)^*\mathcal{T}_\mathfrak{X}
\end{CD}
\end{equation}

On the other hand, the inverse image functor under the map $\phi\circ\psi$ is defined by the pullback diagram
\[
\begin{CD}
(\phi\circ\psi)^+\mathcal{E} @>{\widetilde{d(\phi\circ\psi)}}>> (\phi\circ\psi)^*\mathcal{E} \\
@V{(\phi\circ\psi)^+(\pi)}VV @VV{(\phi\circ\psi)^*(\pi)}V \\
\mathcal{T}_\mathfrak{Z} @>{d(\phi\circ\psi)}>> (\phi\circ\psi)^*\mathcal{T}_\mathfrak{X}
\end{CD}
\]
Since $\psi^*(d\phi)\circ d\psi = d(\phi\circ\psi)$ the universal property of pull-back provides the canonical map
\begin{equation}\label{transitivity of inverse image}
c^+_{\phi,\psi} \colon \psi^+\phi^+\mathcal{E} \to (\phi\circ\psi)^+\mathcal{E}
\end{equation}
which satisfies
\[
\widetilde{d(\phi\circ\psi)}\circ c^{+}_{\phi,\psi}=\psi^*(\widetilde{d\phi})\circ\widetilde{d\psi}
\]
\[
(\phi\circ\psi)^*(\pi)\circ c^+_{\phi,\psi}=\psi^+\phi^+(\pi).
\]
In general the canonical map $c^{+}_{\phi,\psi}$ is not an isomorphism.

\begin{lemma}\label{lemma: one submersion iso}
Suppose that $(\mathcal{E},\pi) \in \Mod{\mathcal{O}_\mathfrak{X}}/\mathcal{T}_\mathfrak{X}$ is locally free and the maps $\mathcal{T}_\mathfrak{Y} \xrightarrow{d\phi} \phi^* \mathcal{T}_\mathfrak{X} \xleftarrow{\phi^*(\pi)} \phi^*\mathcal{E}$ satisfy condition \eqref{condition C}. Then, the map \eqref{transitivity of inverse image} is an isomorphism.
\end{lemma}
\begin{proof}
The assumptions and Lemma \ref{lemma: condition C} imply that the diagram \eqref{diag: pull-back of inverse image} is Cartesian. Thus, both small squares in \eqref{diag: inverse image of inverse image} are Cartesian hence so is their composition.
\end{proof}

Suppose given yet another map $\xi \colon \mathfrak{W} \to \mathfrak{Z}$. In this case both compositions
\[
\mathfrak{W}\xrightarrow{\psi\circ\xi}\mathfrak{Y}\xrightarrow{\phi}\mathfrak{X},\;\;
\mathfrak{W}\xrightarrow{\xi}\mathfrak{Z}\xrightarrow{\phi\circ\psi}\mathfrak{X}
\]
provide respectively the canonical maps

\begin{equation}\label{composition-1}
c^{+}_{\phi,\psi\circ\xi} \colon (\psi\circ\phi)^{+}\phi^+\mathcal{E} \to (\phi\circ\psi\circ\xi)^{+}\mathcal{E}
\end{equation}

\begin{equation}\label{composition-2}
c^{+}_{\phi\circ\psi,\xi} \colon \xi^+(\phi\circ\psi)^+\mathcal{E} \to (\phi\circ\psi\circ\xi)^+\mathcal{E}
\end{equation}
These two maps satisfy similar compatibility conditions as the map \eqref{transitivity of inverse image}, namely
\[
\widetilde{d(\phi\circ\psi\circ\xi)}\circ c^{+}_{\phi,\psi\circ\xi} = (\psi\circ\xi)^{*}(\widetilde{d\phi})\circ\widetilde{d(\psi\circ\xi)}
\]
\[
(\phi\circ\psi\circ\xi)^+(\pi)\circ c^{+}_{\phi,\psi\circ\xi} = (\psi\circ\xi)^+\phi^+(\pi)
\]
and
\[
\widetilde{d(\phi\circ\psi\circ\xi)}\circ c^{+}_{\phi\circ\psi,\xi} = \xi^{*}(\widetilde{d(\phi\circ\psi)})\circ\widetilde{d(\xi)}
\]
\[
(\phi\circ\psi\circ\xi)^+(\pi)\circ c^{+}_{\phi\circ\psi,\xi} = \xi^+(\phi\circ\psi)^+(\pi)
\]
Applying the functor $\xi^{+}$ to the map \eqref{transitivity of inverse image} we obtain the map

\begin{equation}\label{triple composition of inverse image}
\xi^+(c^+_{\phi,\psi}) \colon \xi^+\psi^+\phi^+\mathcal{E} \to \xi^+(\phi\circ\psi)^+\mathcal{E}
\end{equation}
which satisfies
\[
\xi^+\psi^+\phi^+(\pi)=\xi^+(\phi\circ\psi)^+(\pi)\circ\xi^+(c^+_{\phi,\psi}).
\]
The inverse image functor $\xi^+$ on the object $\left(\psi^+\phi^+\mathcal{E},\psi^+\phi^+(\pi)\right)\in\Mod{\mathcal{O}_\mathfrak{Z}}/\mathcal{T}_\mathfrak{Z}$ is defined by the pullback square
\[
\begin{CD}
\xi^+\psi^+\phi^+\mathcal{E} @>{\widetilde{d\xi}}>> \xi^*\psi^+\phi^+\mathcal{E} \\
@V{\xi^+\psi^+\phi^+(\pi)}VV @VV{\xi^*\psi^+\phi^+(\pi)}V \\
\mathcal{T}_{\mathfrak{W}} @>{d\xi}>> \xi^*\mathcal{T}_{\mathfrak{Z}}
\end{CD}
\]
The commutative diagram,
\[
\begin{tikzcd}
\xi^+\psi^+\phi^+\mathcal{E}\arrow[r, "\widetilde{d\xi}"]\arrow[d, "\xi^+\psi^+\phi^+(\pi)"'] & \xi^*\psi^+\phi^+\mathcal{E} \arrow[d, "\xi^*\psi^+\phi^+(\pi)"] \arrow[r, "\xi^{*}(\widetilde{d\psi})"] & \xi^*\psi^*\phi^+\mathcal{E}\arrow[r, "\xi^*\psi^*\widetilde{(d\phi)}"]\arrow[d, "\xi^*\psi^*\phi^+(\pi)"] & \xi^*\psi^*\phi^*\mathcal{E}\cong(\phi\circ\psi\circ\xi)^*\mathcal{E} \arrow[d, "(\phi\circ\psi\circ\xi)^{*}(\pi)"] \\
\mathcal{T}_{\mathfrak{W}}\arrow[r, "d\xi"'] & \xi^*\mathcal{T}_{\mathfrak{Z}} \arrow[r, "\xi^*(d\psi)"'] & \xi^*\psi^*\mathcal{T}_{\mathfrak{Y}}\arrow[r, "\xi^*\psi^*(d\phi)"'] & \xi^*\psi^*\phi^*\mathcal{T}_{\mathfrak{X}}\cong(\phi\circ\psi\circ\xi)^*\mathcal{T}_{\mathfrak{X}}
\end{tikzcd}
\]
and the universal property of pullback provide the canonical map
\[
\xi^+\psi^+\phi^+\mathcal{E} \to (\phi\circ\psi\circ\xi)^+\mathcal{E}.
\]

\begin{lemma}
The identity
\[
c^+_{\phi\circ\psi,\xi}\circ \xi^+(c^+_{\phi,\psi}) = c^+_{\phi,\psi\circ\xi}\circ c^+_{\psi,\xi} \colon \xi^+\psi^+\phi^+\mathcal{E} \to (\phi\circ\psi\circ\xi)^+\mathcal{E}
\]
holds.
\end{lemma}
\begin{proof}
It is enough to check that both compositions make the following diagram
\[
\begin{tikzcd}
\xi^+\psi^+\phi^+\mathcal{E}\arrow[drr, bend left,"\xi^*\psi^*(\widetilde{d\phi})\circ\xi^*(\widetilde{d\psi})\circ\widetilde{d\xi}"]\arrow[ddr, bend right, "\xi^+\psi^+\phi^+(\pi)"']\arrow[dr, dotted] & & \\
 & (\phi\circ\psi\circ\xi)^+\mathcal{E}\arrow[r,"\widetilde{d(\phi\circ\psi\circ\xi)}"]\arrow[d,"(\phi\circ\psi\circ\xi)^+(\pi)"] & (\phi\circ\psi\circ\xi)^*\mathcal{E}\arrow[d, "(\phi\circ\psi\circ\xi)^*(\pi)"] \\
 & \mathcal{T}_{\mathfrak{W}}\arrow[r, "d(\phi\circ\psi\circ\xi)"'] & (\phi\circ\psi\circ\xi)^*\mathcal{T}_{\mathfrak{X}}
\end{tikzcd}
\]
commutative. The computations
\begin{eqnarray*}
\widetilde{d(\phi\circ\psi\circ\xi)}\circ c^+_{\phi,\psi\circ\xi}\circ c^+_{\psi,\xi} & = & \xi^*\psi^*(\widetilde{d\phi})\circ\widetilde{d(\psi\circ\xi)}\circ c^+_{\psi,\xi} \\
 & = & (\psi\circ\xi)^{*}(\widetilde{d\phi})\circ\xi^{*}\widetilde{d\xi}
\end{eqnarray*}
and
\begin{eqnarray*}
\widetilde{d(\phi\circ\psi\circ\xi)}\circ c^+_{\phi\circ\psi,\xi}\circ \xi^+(c^+_{\phi,\psi}) & = & \xi^*(\widetilde{d(\phi\circ\psi)})\circ\widetilde{d\xi}\circ \xi^+(c^+_{\phi,\psi}) \\
& = & \xi^*\psi^*(\widetilde{d\phi})\circ\xi^*(\widetilde{d\psi})\circ\widetilde{d\xi}
\end{eqnarray*}
show that $\widetilde{d(\phi\circ\psi\circ\xi)}\circ c^+_{\phi,\psi\circ\xi}\circ c^+_{\psi,\xi} =\widetilde{d(\phi\circ\psi\circ\xi)}\circ c^+_{\phi\circ\psi,\xi}\circ \xi^+(c^+_{\phi,\psi})$. Similarly, the calculations
\begin{eqnarray*}
(\phi\circ\psi\circ\xi)^+(\pi)\circ c^+_{\phi,\psi\circ\xi}\circ c^+_{\psi,\xi} & = & (\psi\circ\xi)^+\circ\phi^+(\pi)\circ c^+_{\psi,\xi} \\
 & = & \xi^+\psi^+\phi^+(\pi)
\end{eqnarray*}
and
\begin{eqnarray*}
(\phi\circ\psi\circ\xi)^+(\pi)\circ c^+_{\phi\circ\psi,\xi}\circ \xi^+(c^+_{\phi,\psi} & = &  \xi^+(\phi\circ\psi)^+(\pi)\circ c^+_{\psi,\xi} \\
 & = & \xi^+\psi^+\phi^+(\pi)
\end{eqnarray*}
show that $(\phi\circ\psi\circ\xi)^+(\pi)\circ c^+_{\phi,\psi\circ\xi}\circ c^+_{\psi,\xi} = (\phi\circ\psi\circ\xi)^+(\pi)\circ c^+_{\phi\circ\psi,\xi}\circ \xi^+(c^+_{\phi,\psi})$. Therefore $c^+_{\phi\circ\psi,\xi}\circ \xi^+(c^+_{\phi,\psi}) = c^+_{\phi,\psi\circ\xi}\circ c^+_{\psi,\xi}$.
\end{proof}

\section{Localization and descent}\label{section: Localization and descent}
Our exposition is inspired by \citep{BB}.

\subsection{General framework}\label{subsection: general framework}
We denote by $\Man$ the category of manifolds. For $X \in \Man$ we denote by $X_{sm}$ the full subcategory of $\Man/X$ with objects $P \xrightarrow{\pi_P} X$, where $\pi_P$ is a submersion. The category $X_{sm}$ is equipped with the Grothendieck topology (the ``smooth" topology) with covers given by surjective submersions.

We assume given a category $\mathcal{P}$ equipped with a functor $\mathtt{p} \colon \mathcal{P} \to \Man$;  for a manifold $X$ we denote by $\mathcal{P}_X$ the corresponding fiber, i.e. the full subcategory with objects $\mathcal{F}\in\mathcal{P}$ such that $\mathtt{p}(\mathcal{F}) = X$. We assume that the functor $\mathtt{p}$ makes $\mathcal{P}$ a category prefibered over $\Man$. This means, by definition, that
\begin{itemize}
\item For a morphism $Y \xrightarrow{f} X$ in $\Man$ there is a functor $f^\blacktriangledown \colon \mathcal{P}_X \to \mathcal{P}_Y$ of inverse image along $f$

\item For a pair of composable morphisms $Z \xrightarrow{g} Y \xrightarrow{f} X$ there is a morphism of functors $c^\blacktriangledown_{f,g} \colon g^\blacktriangledown f^\blacktriangledown \to (f\circ g)^\blacktriangledown$
\end{itemize}
which satisfy the associativity constraint
\[
c^\blacktriangledown_{f\circ g,h}\circ h^\blacktriangledown(c^\blacktriangledown_{f,g}) = c^\blacktriangledown_{f,g\circ h}\circ c^\blacktriangledown_{g,h} \colon h^\blacktriangledown g^\blacktriangledown f^\blacktriangledown\mathcal{F} \to (f\circ g\circ h)^\blacktriangledown\mathcal{F} \ ,
\]
for any triple of composable morphisms $W \xrightarrow{h} Z \xrightarrow{g} Y \xrightarrow{f} X$ and any $\mathcal{F}\in\mathcal{P}_X$.

In addition, for each manifold $X$ we assume given a full subcategory $\mathcal{P}^\flat_X$ of $\mathcal{P}_X$. We assume that, if $Y \xrightarrow{f} X$ is a submersion and $\mathcal{F} \in \mathcal{P}^\flat_X$ then $f^\blacktriangledown\mathcal{F} \in \mathcal{P}^\flat_Y$, and for $g$ arbitrary the morphism  $c^\blacktriangledown_{f,g} \colon g^\blacktriangledown f^\blacktriangledown\mathcal{F} \to (f\circ g)^\blacktriangledown\mathcal{F}$ is an isomorphism.

\begin{example}{\ }
\begin{enumerate}
\item Let $\left(\Mod{\mathcal{O}}/\mathcal{T}\right)^\sharp$ denote the category with objects pairs $(X,\mathcal{E})$, where $X$ is a manifold and $\mathcal{E} \in \Mod{\mathcal{O}_{X^\sharp}}/\mathcal{T}_{X^\sharp}$. A morphism $u \colon (Y,\mathcal{F}) \to (X,\mathcal{E})$ is a pair $u=(f, t)$, where $f \colon Y \to X$ is a map of manifolds and $t \colon \mathcal{F} \to f^{\sharp +}\mathcal{E}$ is a morphism in $\Mod{\mathcal{O}_{X^\sharp}}/\mathcal{T}_{X^\sharp}$.

The functor
\begin{equation}\label{mods over T to man}
\left(\Mod{\mathcal{O}}/\mathcal{T}\right)^\sharp \to \Man \colon
(X,\mathcal{E}) \mapsto X, (f, t) \mapsto f
\end{equation}
makes $\left(\Mod{\mathcal{O}}/\mathcal{T}\right)^\sharp$ a prefibered category over the category $\Man$ of manifolds.

Let $\left(\Mod{\mathcal{O}}^\mathtt{lf}/\mathcal{T}\right)^\sharp_X$ denote the full subcategory of $\left(\Mod{\mathcal{O}}/\mathcal{T}\right)^\sharp_X$ with objects locally free of finite rank over $\mathcal{O}_{X^\sharp}$.

This is an example of the framework of \ref{subsection: general framework} with $\mathcal{P} = \left(\Mod{\mathcal{O}}/\mathcal{T}\right)^\sharp$, $\mathcal{P}^\flat_X = \left(\Mod{\mathcal{O}}^\mathtt{lf}/\mathcal{T}\right)^\sharp_X$ and the functor of inverse image defined in \ref{subsection: inverse image}.

\item Lie algebroids, see \ref{subsection: Smooth localization for Lie algebroids}

\item Marked Lie algebroids, see \ref{subsection: Smooth localization for marked Lie algebroids}

\item Courant algebroids, see \ref{subsection: Smooth localization for Courant algebroids}
\end{enumerate}

\end{example}

\subsection{The category of descent data}
For $P \xrightarrow{\pi_P} X$ in $X_{sm}$ consider the diagram
\[
P\times_X P\times_X P \triplerightarrow{\pi_{ij}}{} P\times_X P  \doublerightarrow{\pi_i}{}
P \xrightarrow{\pi_P}  X
\]
where the maps $\pi_i$, $i=1,2$ (respectively, $\pi_{ij}$, $1\leqslant i < j \leqslant 3$) denote projections onto the $i^{\text{th}}$  (respectively, $i^{\text{th}}$ and $j^{\text{th}}$) factor.

We denote by $\Desc(P;\mathcal{P}^\flat) = \Desc(P \xrightarrow{\pi_P} X;\mathcal{P}^\flat)$ the category with objects pairs $(\mathcal{F}, g_\mathcal{F})$, where
\begin{itemize}
\item $\mathcal{F} \in \mathcal{P}^\flat_P$;

\item $g_\mathcal{F} \colon \pi_2^\blacktriangledown\mathcal{F} \to \pi_1^\blacktriangledown\mathcal{F}$ is an isomorphism which satisfies the cocycle condition $\pi_{12}^\blacktriangledown(g_\mathcal{F})\circ\pi_{23}^\blacktriangledown(g_\mathcal{F}) = \pi_{13}^\blacktriangledown(g_\mathcal{F})$.
\end{itemize}
A morphism $t \colon (\mathcal{F}, g_\mathcal{F}) \to (\mathcal{F}^\prime, g_{\mathcal{F}^\prime})$ in $\Desc(P;\mathcal{P}^\flat)$ is a morphism $t \colon \mathcal{F} \to \mathcal{F}^\prime$ in $\mathcal{P}^\flat_P$ which satisfies $\pi_1^{\prime\blacktriangledown}(t) \circ g_{\mathcal{F}^\prime} = g_{\mathcal{F}} \circ \pi_2^\blacktriangledown(t)$.

The assignment $(\mathcal{F}, g_\mathcal{F}) \mapsto \mathcal{F}$ defines a functor $\Desc(P;\mathcal{P}^\flat) \to \mathcal{P}^\flat_P$.

\subsection{Localization}
For $\mathcal{E}\in\mathcal{P}^\flat_X$ and $P \xrightarrow{\pi_P} X$ in $X_{sm}$ let $\mathcal{E}_P := \pi_P^\blacktriangledown\mathcal{E} \in \mathcal{P}^\flat_P$. If $Q \xrightarrow{f} P$ is a morphism in $X_{sm}$, the map $c_{\pi_P,f} \colon \mathcal{E}_Q \to \mathcal{E}_P$ is an isomorphism. Therefore, the assignment $X_{sm} \ni P \mapsto \mathcal{E}_P$ defines a Cartesian section of $\mathcal{P}/X_{sm}$.

The functor $\pi_P^\blacktriangledown \colon \mathcal{P}^\flat_X \to \mathcal{P}^\flat_P$ lifts to the functor
\begin{equation}\label{localization to descent data}
\widetilde{\pi_P^\blacktriangledown} \colon \mathcal{P}_X \to \Desc(P;\mathcal{P}^\flat)
\end{equation}
defined by $\mathcal{E} \mapsto (\mathcal{E}_P, g_{\mathcal{E}_P})$, where $g_{\mathcal{E}_P}$ is the composition of the canonical isomorphisms $\pi_2^\blacktriangledown\mathcal{E}_P \cong \mathcal{E}_{P\times_X P} \cong \pi_1^\blacktriangledown\mathcal{E}_P$.

A morphism $f \colon Q \to P$ in $X_{sm}$ induces the functor
\[
\widetilde{f^\blacktriangledown} \colon \Desc(P;\mathcal{P}^\flat) \to \Desc(Q;\mathcal{P}^\flat)
\]
defined by $(\mathcal{F}, g_\mathcal{F}) \mapsto (f^\blacktriangledown\mathcal{F}, (f\times_X f)^\blacktriangledown g_\mathcal{F})$.

\subsection{Descent}
We shall say that $P\xrightarrow{\pi_P} X \in X_{sm}$ is \emph{classical} if $\pi_P$ is a disjoint union of open embeddings.

For any smooth cover $P\xrightarrow{\pi_P} X$ there exists a classical cover $T\xrightarrow{\pi_T} X$ such that $\Hom_{X_{sm}}(T,P) \neq \varnothing$, i.e. there exists a morphism $f \colon T \to P$ such that $\pi_T = f\circ\pi_P$. In other words, every smooth cover admits a classical refinement.

We shall say that $\mathcal{P}^\flat$ has the smooth (respectively, classical) descent property if for any smooth (respectively, classical) cover $P \xrightarrow{\pi_P} X$ the functor \eqref{localization to descent data} is an equivalence.

\begin{theorem}\label{thm: smooth descent}
If $\mathcal{P}^\flat$ has the classical descent property then it has the smooth descent property.
\end{theorem}
\begin{proof}
Suppose that $f \colon T \to P$ is a morphism in $X_{sm}$ with $T\xrightarrow{\pi_T} X$ a classical cover. For $(\mathcal{F}, g_\mathcal{F}) \in \Desc(P;\mathcal{P}^\flat)$ there exist $\mathcal{E} \in \mathcal{P}^\flat_X$ and an isomorphism $\widetilde{\pi_T^\blacktriangledown}\mathcal{E} \cong \widetilde{f^\blacktriangledown}(\mathcal{F}, g_\mathcal{F})$.

Consider the diagram
\begin{equation}\label{diag: pull-back cover}
\begin{CD}
P\times_X T @>{\pr_T}>> T \\
@V{\pr_P}VV @VV{\pi_T}V \\
P @>{\pi_P}>> X
\end{CD}
\end{equation}
Note that $\pr_P \colon P\times_X T \to P$ a classical cover of $P$. Since $\pi_P\circ\pr_P = \pi_T\circ\pr_T$, $\pr_2\circ(\id\times f) = f\circ\pr_T$ and $\pr_1\circ(\id\times f) = \pr_P$ there are isomorphisms
\begin{multline}
\widetilde{\pr_P^\blacktriangledown}\widetilde{\pi_P^\blacktriangledown}\mathcal{E} \cong \widetilde{\pr_T^\blacktriangledown}\widetilde{\pi_T^\blacktriangledown}\mathcal{E} \cong \widetilde{\pr_T^\blacktriangledown}\widetilde{f^\blacktriangledown}(\mathcal{F}, g_\mathcal{F})\cong \\ \cong\widetilde{(\id\times f)^\blacktriangledown}\widetilde{\pr_2^\blacktriangledown}(\mathcal{F}, g_\mathcal{F}) \xrightarrow{\widetilde{(\id\times f)^\blacktriangledown}(g_\mathcal{F})} \widetilde{(\id\times f)^\blacktriangledown}\widetilde{\pr_1^\blacktriangledown}(\mathcal{F}, g_\mathcal{F})
\cong \widetilde{\pr_P^\blacktriangledown}(\mathcal{F}, g_\mathcal{F})
\end{multline}
Since $\pr_P \colon P\times_X T \to P$ is a classical cover of $P$ the functor $\widetilde{\pr_P^\blacktriangledown}$ is an equivalence by assumption. Therefore, an isomorphism $\widetilde{\pr_P^\blacktriangledown}(\mathcal{F}, g_\mathcal{F}) \cong \widetilde{\pr_P^\blacktriangledown}\widetilde{\pi_P^\blacktriangledown}\mathcal{E}$ is equivalent to an isomorphism $(\mathcal{F}, g_\mathcal{F}) \cong \widetilde{\pi_P^\blacktriangledown}\mathcal{E}$. Thus, the functor $\widetilde{\pi_P^\blacktriangledown}$ is essentially surjective.

Since $\widetilde{f^\blacktriangledown}\circ\widetilde{\pi_P^\blacktriangledown} \cong \widetilde{\pi_T^\blacktriangledown}$ is an equivalence, it follows that $\widetilde{\pi_P^\blacktriangledown}$ is faithful.

Suppose that $(\mathcal{F}_1, g_{\mathcal{F}_1}) \xrightarrow{\phi}(\mathcal{F}, g_{\mathcal{F}_2})$ is a morphism in $\Desc(P;\mathcal{P}^\flat)$. Then, there exits a morphism $\mathcal{E}_1 \xrightarrow{\psi}\mathcal{E}_2$ in $\mathcal{P}^\flat_X$ and isomorphisms $\widetilde{\pi_T^\blacktriangledown}\mathcal{E}_i \cong \widetilde{f^\blacktriangledown}(\mathcal{F}_i, g_{\mathcal{F}_i})$ which intertwine $\phi$ and $\psi$. Proceeding as in the proof of essential subjectivity one obtains the commutative diagram
\[
\begin{CD}
\widetilde{\pr_P^\blacktriangledown}\widetilde{\pi_P^\blacktriangledown}\mathcal{E}_1 @>{\cong}>> \widetilde{\pr_P^\blacktriangledown}(\mathcal{F}_1, g_{\mathcal{F}_1}) \\
@V{\widetilde{\pr_P^\blacktriangledown}\widetilde{\pi_P^\blacktriangledown}(\psi)}VV @VV{\widetilde{\pr_P^\blacktriangledown}(\phi)}V \\
\widetilde{\pr_P^\blacktriangledown}\widetilde{\pi_P^\blacktriangledown}\mathcal{E}_2 @>{\cong}>> \widetilde{\pr_P^\blacktriangledown}(\mathcal{F}_1, g_{\mathcal{F}_2})
\end{CD}
\]
Since the functor $\widetilde{\pr_P^\blacktriangledown}$ is an equivalence by assumption it follows that $\phi = \widetilde{\pi_P^\blacktriangledown}(\psi)$. Therefore the functor $\widetilde{\pi_P^\blacktriangledown}$ is full.
\end{proof}
%
%
%
%

\begin{corollary}[of Theorem \ref{thm: smooth descent}]
$\left(\Mod{\mathcal{O}}^\mathtt{lf}/\mathcal{T}\right)^\sharp$ has the smooth descent property.
\end{corollary}

\section{Inverse image for Lie algebroids}\label{section: Inverse image for Lie algebroids}
The anchor map of an $\mathcal{O}_{\mathfrak{X}}$-Lie algebroid
$\mathcal{A}$ renders the latter as an object of
$\Mod{\mathcal{O}_\mathfrak{X}}/\mathcal{T}_\mathfrak{X}$. This
correspondence extends to a functor
$\LA{\mathcal{O}_{\mathfrak{X}}} \to
\Mod{\mathcal{O}_{\mathfrak{X}}}/\mathcal{T}_{\mathfrak{X}}$.

\subsection{Inverse image for Lie algebroids}\label{subsection: inverse image Lie algd}
For an $\mathcal{O}_{\mathfrak{X}}$-Lie algebroid $\mathcal{A}$
and a map of DG-manifolds $\phi \colon \mathfrak{Y} \to \mathfrak{X}$ the inverse
image $\phi^+\mathcal{A}$ has a canonical structure of a Lie
algebroid on $\mathfrak{Y}$. Namely, the bracket on
$\phi^+\mathcal{A}$ is given by
\[
[(f\otimes a,\xi), (g\otimes b, \eta)] = ((-1)^{ag}fg\otimes [a,b] +
\xi(g)\otimes b - (-1)^{\eta\xi}\eta(f)\otimes a, [\xi,\eta]) ,
\]
where $f,g\in\mathcal{O}_{\mathfrak{Y}}$, $a,b\in\mathcal{A}$,
$\xi, \eta\in \mathcal{T}_{\mathfrak{Y}}$ are homogeneous elements respectively. Since $(f\otimes a,\xi),\ (g\otimes b, \eta)$ are homogeneous elements, $a+f=\xi$ and $g+b=\eta$.

For any composition of maps of DG-manifolds
$$\mathfrak{Z}\xrightarrow{\psi} \mathfrak{Y}\xrightarrow{\phi}\mathfrak{X},$$
a general element in $\psi^+\phi^+\mathcal{A}$ is a finite sum of homogeneous
elements of the form $(h\otimes(f\otimes a,\xi),\rho)$, where
$h\in\mathcal{O}_{\mathfrak{Z}}$,
$f\in\mathcal{O}_{\mathfrak{Y}}$, $a\in\mathcal{A}$, $\xi \in
\mathcal{T}_{\mathfrak{Y}}$ and $\rho\in
\mathcal{T}_{\mathfrak{Z}}$. In these terms, the canonical map
$c^+_{\phi,\psi}:\psi^+\phi^+\mathcal{A}\to(\phi\circ\psi)^+\mathcal{A}$
in the category
$\Mod{\mathcal{O}_\mathfrak{Z}}/\mathcal{T}_\mathfrak{Z}$ is equal
to
$$c^+_{\phi,\psi}(h\otimes(f\otimes a,\xi),\rho)=(h\psi^{*}(f)\otimes a
,\rho),$$ where,
\begin{equation}\label{rel-pull}
(d\psi)(\rho)=h\otimes\xi,
\end{equation}
see Subsection  \ref{Inverse and composition}.
\begin{lemma}\label{lemma: cfg lie algd}
The map $c^+_{\phi,\psi}$ is a morphism of
$\mathcal{O}_{\mathfrak{Z}}$-Lie algebroids.
\end{lemma}
\begin{proof}
\begin{multline*}
[c^+_{\phi,\psi}(h\otimes(f\otimes a,\xi),\rho),
c^+_{\phi,\psi}(l\otimes(g\otimes b,\eta),\nu)]
\\
= [(h\psi^{*}(f)\otimes a ,\rho),(l\psi^{*}(g)\otimes
b ,\nu)]
\\
= (\rho(l)\psi^*(g)\otimes b +(-1)^{\rho l} l\rho(\psi^*(g))\otimes
b-(-1)^{\nu\rho}\nu(h)\psi^*(f)\otimes a
\\
 -(-1)^{\nu\rho+\nu h} h\nu(\psi^*(f))\otimes a + (-1)^{lf+a(f+g)}
hl\psi^*(fg)\otimes[a,b] ,[\rho,\nu]).
\end{multline*}
On the other hand,
\begin{multline*}
c^+_{\phi,\psi}[(h\otimes(f\otimes a,\xi),\rho),
(l\otimes(g\otimes b,\eta),\nu)]
\\
 = c^+_{\phi,\psi}(\rho(l)\otimes(g\otimes
b,\eta)-(-1)^{\nu\rho}\nu(h)\otimes(f\otimes a,\xi)+(-1)^{l\xi}hl\otimes[(f\otimes
a,\xi),(g\otimes b, \eta)],[\rho,\nu])
\\
 = c^+_{\phi,\psi}(\rho(l)\otimes(g\otimes
b,\eta)-(-1)^{\nu\rho}\nu(h)\otimes(f\otimes a,\xi) \\
+(-1)^{l\xi}hl\otimes((-1)^{ag}fg\otimes [a,b] +
\xi(g)\otimes b - (-1)^{\eta\xi}\eta(f)\otimes a, [\xi,\eta]),[\rho,\nu])
\\
= (\rho(l)\psi^*(g)\otimes b -(-1)^{\nu\rho}\nu(h)\psi^*(f)\otimes a+(-1)^{l\xi}
hl(\psi^*(\xi(g))\otimes b \\
-(-1)^{l\xi+\eta\xi}\psi^*(\nu(f))\otimes
a+(-1)^{l\xi+ag}\psi^*(fg)\otimes[a,b]),[\rho,\nu]).
\end{multline*}
Comparing term by term the calculations of
$[c^+_{\phi,\psi}(h\otimes(f\otimes a,\xi),\rho),
c^+_{\phi,\psi}(l\otimes(g\otimes b,\eta),\nu)]$ and
$c^+_{\phi,\psi}[(h\otimes(f\otimes a,\xi),\rho),
(l\otimes(g\otimes b,\eta),\nu)]$, it remains to show that
\[
hl(\psi^*(\xi (g)))\otimes b = l\rho(\psi^*(g))\otimes b,
\]
\[
hl(\psi^*(\eta (f)))\otimes a = h\nu(\psi^*(f))\otimes a.
\]
Applying the formula \eqref{rel-pull} in the first case one gets,
\[
(-1)^{l\xi}hl(\psi^*(\xi (g)))\otimes b=(-1)^{lh+l\xi}l(h\psi^*(\xi (g)))\otimes
b=(-1)^{l\rho}l(h\otimes\xi(g)\otimes b)=l\rho(\psi^*(g))\otimes b.
\]
The second case is analogous.
\end{proof}

\subsection{Inverse image for marked Lie algebroids}\label{subsection: Inverse image for marked Lie algebroids}
Suppose that $\mathfrak{Y}\xrightarrow{\phi}\mathfrak{X}$ is a map of DG-manifolds. Let $(\mathcal{A},\mathfrak{c})\in\LA{\mathcal{O}_{\mathfrak{X}}}^\star_n$ be a marked Lie algebroid on $\mathfrak{X}$, i.e. there is a map of marked Lie algebroids
\[
\mathcal{O}_{\mathfrak{X}}[n]\xrightarrow{\cdot \mathfrak{c}}\mathcal{A}. 
\]
Since $\phi^+\mathcal{O}_{\mathfrak{X}}[n]=\ker(d\phi)\oplus\mathcal{O}_{\mathfrak{Y}}[n]$, there is a canonical map
\begin{equation}\label{marking map}
\mathcal{O}_{\mathfrak{Y}}[n]\to\ker(d\phi)\oplus\mathcal{O}_{\mathfrak{Y}}[n] = \phi^+\mathcal{O}_{\mathfrak{X}}[n]
\end{equation}
Let $\phi^{+}(\mathfrak{c}) \in \Gamma(Y;\phi^+\mathcal{A})^{-n}$ denote the image of $1 \in \Gamma(Y;\mathcal{O}_{\mathfrak{Y}})$ under the composition
\[
\mathcal{O}_{\mathfrak{Y}}[n] \xrightarrow{\eqref{marking map}} \phi^+\mathcal{O}_{\mathfrak{X}}[n] \xrightarrow{\phi^+(\cdot \mathfrak{c})} \phi^+\mathcal{A}\ .
\]
The pair $(\phi^+\mathcal{A}, \phi^{+}(\mathfrak{c}))$ is a marked Lie algebroid denoted $\phi^+(\mathcal{A},\mathfrak{c})$. The map
\[
\mathcal{O}_{\mathfrak{Y}}[n]\xrightarrow{\cdot \phi^{+}(\mathfrak{c})}\phi^+\mathcal{A}\ .
\]
is given by $g\mapsto (g\otimes \mathfrak{c},0)$ for $g\in\mathcal{O}_{\mathfrak{Y}}[n]$.

The assignment $(\mathcal{A},\mathfrak{c}) \mapsto \phi^+(\mathcal{A},\mathfrak{c})$ extends to a functor
\[
\phi^+ \colon \LA{\mathcal{O}_{\mathfrak{X}}}^\star_n \to \LA{\mathcal{O}_{\mathfrak{Y}}}^\star_n\ .
\]

For any composition of maps of DG-manifolds, $\mathfrak{Z}\xrightarrow{\psi} \mathfrak{Y}\xrightarrow{\phi}\mathfrak{X}$, the morphism $c^{+}_{\phi,\psi}$ makes the diagram
\[
\xymatrix{ \mathcal{O}_{\mathfrak{Z}}[n] \ar[rrd]_{\cdot \psi^+\phi^+(\mathfrak{c})}\ar[rr]^{\cdot (\phi\circ\psi)^+(\mathfrak{c})} & & (\phi\circ\psi)^+\mathcal{A} \\ & & \psi^+\phi^+\mathcal{A}\ar[u]_{c^+_{\phi,\psi}} \\
}
\]
commutative.

\begin{lemma}\label{lemma: preservation of marking}
The map $c^+_{\phi,\psi}$ is a morphism of marked Lie algebroids.
\end{lemma}
\begin{proof}
Lemma \ref{lemma: cfg lie algd} implies that $c^+_{\phi,\psi}$ is a morphism of $\mathcal{O}_{\mathfrak{Y}}$-Lie algebroids.  For $g\in\mathcal{O}_{\mathfrak{Z}}$
\[
g \cdot (\phi\circ\psi)^+(\mathfrak{c}) = (g\otimes\mathfrak{c},0),
\]
by construction. On the other hand, the calculation
\[
c^+_{\phi,\psi}(g\cdot \psi^+\phi^+(\mathfrak{c})) = c^+_{\phi,\psi}(g\otimes(1\otimes \mathfrak{c},0),0)=(g\otimes\mathfrak{c},0)
\]
shows that $g\cdot(\phi\circ\psi)^+(\mathfrak{c}) = c^{+}_{\phi,\psi}(g\cdot\psi^+\psi^+(\mathfrak{c}))$, i.e.  $c^+_{\phi,\psi}$ preserves the marking.
\end{proof}

\subsection{Inverse image for $\mathcal{O}[n]$-extensions}
Suppose that $\mathfrak{Y}\xrightarrow{\phi}\mathfrak{X}$ is a map of DG-manifolds and $\mathcal{B}$ is a $\mathcal{O}_{\mathfrak{X}}$-Lie algebroid.

\begin{lemma}\label{lemma: inverse image preserves extensions}
Suppose that $(\mathcal{A},\mathfrak{c})$ is a $\mathcal{O}[n]$-extension of $\mathcal{B}$. Then, $\phi^+(\mathcal{A},\mathfrak{c})$ is a $\mathcal{O}[n]$-extension of $\phi^+\mathcal{B}$.
\end{lemma}
\begin{proof}
Since both small squares in the commutative diagram
\[
\begin{CD}
\phi^+\mathcal{B}\times_{\phi^*\mathcal{B}}\phi^*\mathcal{A} @>>> \phi^+\mathcal{B} @>>> \mathcal{T}_{\mathfrak{Y}} \\
@VVV @VVV @VVV \\
\phi^*\mathcal{A} @>>> \phi^*\mathcal{B} @>>> \phi^*\mathcal{T}_{\mathfrak{X}}
\end{CD}
\]
are Cartesian, so is the large one. Therefore there is a canonical isomorphism $\phi^+\mathcal{A} \cong \phi^+\mathcal{B}\times_{\phi^*\mathcal{B}}\phi^*\mathcal{A}$. Hence, $\phi^+(\mathcal{A},\mathfrak{c})$ is a $\mathcal{O}[n]$-extension of $\phi^+\mathcal{B}$.
\end{proof}

Thus, the inverse image functor for marked Lie algebroids induces a functor
\begin{equation}\label{inverse image extensions Lie}
\phi^+ \colon \OExt{n}(\mathcal{\mathcal{B}}) \to \OExt{n}(\mathcal{\phi^+\mathcal{B}})
\end{equation}

\begin{proposition}\label{lemma: inverse image is linear}
The functor \eqref{inverse image extensions Lie} is a morphism of $\mathbb{C}$-vector spaces in categories (and, in particular, of Picard groupoids).
\end{proposition}
\begin{proof} We prove the statement in the case $\lambda_1\mathcal{A}_1 \opdp\lambda_2\mathcal{A}_2$. The general case is analogous and is left to the reader

Recall that for any object $\mathcal{A}\in\OExt{n}(\mathcal{\mathcal{B}})$ there is the canonical isomorphism
$\phi^+\mathcal{A} = \phi^+\mathcal{B}\times_{\phi^*\mathcal{B}}\ \phi^*\mathcal{A}$.
Thus, for a pair of objects $\mathcal{A}_1,\mathcal{A}_2 \in \OExt{n}(\mathcal{\mathcal{B}})$ there is a sequence of isomorphisms
\begin{multline*}
\phi^+\mathcal{A}_1\times_{\phi^+\mathcal{B}}\phi^+\mathcal{A}_2 \cong (\phi^+\mathcal{B}\times_{\phi^*\mathcal{B}}\phi^*\mathcal{A}_1)\times_{\phi^+\mathcal{B}}(\phi^+\mathcal{B}\times_{\phi^*\mathcal{B}}\phi^*\mathcal{A}_2) \\
\cong \mathcal{T}_{\mathfrak{Y}}\times_{\phi^*\mathcal{T}_{\mathfrak{X}}}(\phi^*\mathcal{A}_1 \times_{\phi^*B}\phi^*\mathcal{A}_2) \cong \mathcal{T}_{\mathfrak{Y}}\times_{\phi^*\mathcal{T}_{\mathfrak{X}}}\phi^*(\mathcal{A}_1 \times_{B}\mathcal{A}_2) \cong \phi^+(\mathcal{A}_1\times_{\mathcal{B}}\mathcal{A}_2)\ .
\end{multline*}
The inverse image functor $\phi^+$ applied to the commutative diagram
\[
\begin{CD}
\mathcal{O}_{\mathfrak{X}}[n]\times \mathcal{O}_{\mathfrak{X}}[n] @>>> \mathcal{A}_1\times_{\mathcal{B}}\mathcal{A}_2 \\
@V{(\alpha_1,\alpha_2)\mapsto \lambda_1\alpha_1+\lambda_2\alpha_2}VV @VVV \\
\mathcal{O}_{\mathfrak{X}}[n] @>>> \lambda_1\mathcal{A}_1 \opdp \lambda_2\mathcal{A}_2
\end{CD}
\]
and canonical isomorphisms $\phi^+(\mathcal{O}_{\mathfrak{X}}[n]\times \mathcal{O}_{\mathfrak{X}}[n]) \cong  \phi^+\mathcal{O}_{\mathfrak{X}}[n]\times \phi^+\mathcal{O}_{\mathfrak{X}}[n]$ and $\phi^+\mathcal{A}_1\times_{\phi^+\mathcal{B}}\phi^+\mathcal{A}_2 \cong\phi^+(\mathcal{A}_1\times_{\mathcal{B}}\mathcal{A}_2)$ give rise to the commutative diagram
\[
\begin{CD}
\mathcal{O}_{\mathfrak{Y}}[n]\times \mathcal{O}_{\mathfrak{Y}}[n] @>{(\eqref{marking map},\eqref{marking map})}>> \phi^+\mathcal{O}_{\mathfrak{X}}[n]\times \phi^+\mathcal{O}_{\mathfrak{X}}[n] @>>> \phi^+\mathcal{A}_1\times_{\phi^+\mathcal{B}}\phi^+\mathcal{A}_2 \\
@V{(\alpha_1,\alpha_2)\mapsto \lambda_1\alpha_1+\lambda_2\alpha_2}VV @VV{\phi^+((\alpha_1,\alpha_2)\mapsto \lambda_1\alpha_1+\lambda_2\alpha_2)}V @VVV \\
\mathcal{O}_{\mathfrak{Y}}[n] @>{\eqref{marking map}}>> \phi^+\mathcal{O}_{\mathfrak{X}}[n] @>>> \phi^+(\lambda_1\mathcal{A}_1 \opdp \lambda_2\mathcal{A}_2)\ .
\end{CD}
\]
Therefore, there is a unique morphism in $\OExt{n}(\mathcal{\phi^+\mathcal{B}})$
\begin{equation}\label{map of extensions}
\lambda_1\phi^+\mathcal{A}_1\opdp \lambda_2\phi^+\mathcal{A}_2 \to \phi^+(\lambda_1\mathcal{A}_1 \opdp \lambda_2\mathcal{A}_2)\ .
\end{equation}
\end{proof}

\subsection{Smooth localization for Lie algebroids}\label{subsection: Smooth localization for Lie algebroids}
Let $\left(\LA{\mathcal{O}}\right)^\sharp$ denote the category with objects pairs $(X,\mathcal{A})$, where $X$ is a manifold and $\mathcal{A}\in\LA{\mathcal{O}_{X^\sharp}}$. A morphism $u \colon (Y,\mathcal{B}) \to (X,\mathcal{A})$ is a pair $u=(f, t)$, where $f \colon Y \to X$ is a map of manifolds and $t \colon \mathcal{B} \to f^{\sharp +}(\mathcal{A}$ is a morphism in $\LA{\mathcal{O}_{Y^\sharp}}$. It follows from Lemma \ref{lemma: cfg lie algd} that the forgetful functor $(X,\mathcal{A}) \mapsto X$ makes $\left(\LA{\mathcal{O}}\right)^\sharp$ a category prefibered over $\Man$.

Let $\left(\LA{\mathcal{O}}^\mathtt{lf}\right)^\sharp_X$ denote the full subcategory of $\left(\LA{\mathcal{O}}\right)^\sharp_X$ with objects locally free of finite rank over $\mathcal{O}_{X^\sharp}$.

This is an example of the framework of \ref{subsection: general framework} with $\mathcal{P} = \left(\LA{\mathcal{O}}\right)^\sharp$, $\mathcal{P}^\flat_X = \left(\LA{\mathcal{O}}^\mathtt{lf}\right)^\sharp_X$ and the functor of inverse image defined in \ref{subsection: inverse image Lie algd}.

\begin{corollary}[of Theorem \ref{thm: smooth descent}]
$\left(\LA{\mathcal{O}}^\mathtt{lf}\right)^\sharp$ has the smooth descent property.
\end{corollary}

\subsection{Smooth localization for marked Lie algebroids}\label{subsection: Smooth localization for marked Lie algebroids}
Let $\left(\LA{\mathcal{O}}^\star_n\right)^\sharp$ denote the category with objects pairs $(X,(\mathcal{A},\mathfrak{c}))$, where $X$ is a manifold and $(\mathcal{A},\mathfrak{c})\in\LA{\mathcal{O}_{X^\sharp}}^\star_n$. A morphism $u \colon (Y,(\mathcal{B},\mathfrak{b})) \to (X,(\mathcal{A},\mathfrak{c}))$ is a pair $u=(f, t)$, where $f \colon Y \to X$ is a map of manifolds and $t \colon (\mathcal{B},\mathfrak{b}) \to f^{\sharp +}\mathcal{A},\mathfrak{c})$ is a morphism in $\LA{\mathcal{O}_{Y^\sharp}}^\star_n$. It follows from Lemma \ref{lemma: cfg lie algd} and Lemma \ref{lemma: preservation of marking} that the forgetful functor $(X,(\mathcal{A},\mathfrak{c})) \mapsto X$ makes $\left(\LA{\mathcal{O}}^\star_n\right)^\sharp$ a category prefibered over $\Man$.

Let $\left(\LA{\mathcal{O}}^{\star\mathtt{lf}}_n\right)^\sharp_X$ denote the full subcategory of $\left(\LA{\mathcal{O}}^\star_n\right)^\sharp_X$ with objects locally free of finite rank over $\mathcal{O}_{X^\sharp}$.

This is an example of the framework of \ref{subsection: general framework} with $\mathcal{P} = \left(\LA{\mathcal{O}}^\star_n\right)^\sharp$, $\mathcal{P}^\flat_X = \left(\LA{\mathcal{O}}^{\star\mathtt{lf}}_n\right)^\sharp_X$ and the functor of inverse image defined in, \ref{subsection: Inverse image for marked Lie algebroids}.

\begin{corollary}[of Theorem \ref{thm: smooth descent}]
$\left(\LA{\mathcal{O}}^\star_n\right)^\sharp$ has the smooth descent property.
\end{corollary}

\section{Inverse image for Courant algebroids}\label{section: Inverse image for Courant algebroids}
\subsection{The inverse image functor}\label{subsection: inverse image courants}
Suppose that $f \colon Y \to X$ is a map of manifolds. For $\mathcal{Q}\in \CA(X)$ the functor of inverse image under $f$ for Courant algebroids denoted
\[
f^\ii \colon \CA(X) \to \CA(Y)
\]
is defined by
\[
f^\ii\mathcal{Q} := \Cour(f^+\tau\mathcal{Q}).
\]
in the notation of \ref{subsection: Transgression for Courant algebroids}.

\subsection{Explicit description}
Suppose that $\mathcal{Q}$ is a Courant algebroid on $X$ and $f\colon Y \to X$ is a map of manifolds. It transpires from the definition, that $f^\ii\mathcal{Q}$ is a sub-quotient of $\Omega^1_Y\oplus f^*\mathcal{Q}\oplus\mathcal{T}_Y$:
\[
\Omega^1_Y\oplus f^*\mathcal{Q}\oplus\mathcal{T}_Y \hookleftarrow \Omega^1_Y\oplus f^*\mathcal{Q}\times_{f^*\mathcal{T}_X} \mathcal{T}_Y \to f^\ii\mathcal{Q} .
\]

\subsubsection{The symmetric pairing}
Let $\ip^\prime$ denote the symmetric pairing on $f^*\mathcal{Q}\oplus\mathcal{T}_Y$ such that
\begin{enumerate}
\item it restricts to the pairing induced by $\ip$ on $f^*\mathcal{Q}$,
\item $\langle\mathcal{T}_Y,\mathcal{T}_Y\rangle = 0$,
\item for $\alpha\in f^*\Omega^1_X$, $\xi\in\mathcal{T}_Y$, $\langle i(\alpha),\xi\rangle^\prime = \iota_\xi\dual{df}(\alpha)$.
\end{enumerate}

Let $\ip^{\prime\prime}$ denote the unique symmetric pairing on $\Omega^1_Y\oplus f^*\mathcal{Q}\times_{f^*\mathcal{T}_X} \mathcal{T}_Y$ such that
\begin{enumerate}
\item it restricts to the pairing $\ip^\prime$ on $f^*\mathcal{Q}\times_{f^*\mathcal{T}_X} \mathcal{T}_Y$,
\item $\langle\Omega^1_Y,\Omega^1_Y\rangle = 0$,
\item for $\alpha\in\Omega^1_Y$, $q\in f^*\mathcal{Q}\times_{f^*\mathcal{T}_X} \mathcal{T}_Y$, $\langle \alpha,q\rangle^{\prime\prime} = \iota_{\sigma(q)}\alpha$.
\end{enumerate}

Since
\[
\langle (\dual{df}(\alpha),-i(\alpha),0),(\beta,q,\xi)\rangle^{\prime\prime} = 0
\]
for all $\alpha\in f^*\Omega^1_X$, $\beta\in\Omega^1_Y$, $(q,\xi)\in f^*\mathcal{Q}\times_{f^*\mathcal{T}_X} \mathcal{T}_Y$, it follows that the pairing $\ip^{\prime\prime}$ descends to a symmetric pairing on $f^\ii\mathcal{Q}$ which we will denote by $\ip$.

\subsubsection{The bracket}
Let $[\ ,\ ]^{\prime}$ denote the unique operation on $\Omega^1_Y\oplus f^*\mathcal{Q} \oplus \mathcal{T}_Y$ defined by the formula

\[
[(\alpha, h\otimes q, \xi), (\beta, j\otimes p, \eta)]^{\prime} = (-\iota_{\eta}d\alpha + L_{\xi}\beta + jdh\langle q,p \rangle, hj\otimes [q,p]-\iota_{\eta}dh\otimes q + L_{\xi}j\otimes p, [\xi,\eta]),
\]
where $\alpha,\,\beta\in\Omega^1_Y$, $h,\, j\in\mathcal{O}_Y$, $q,\,p\in\mathcal{Q}$ and $\xi,\eta\in\mathcal{T}_Y$.

For any $\gamma\in f^{*}\Omega^1_X$, the calculation
\begin{multline*}
[(\dual{df}(\gamma),-i(\gamma),0), (\beta, j\otimes p, \eta)]^{\prime}
= (-\iota_{\eta}d(\dual{df}(\gamma)),-j\otimes[\gamma,p],0) \\
= (-\iota_{\eta}d(\dual{df}(\gamma)),j\otimes L_{\pi(p)}\gamma-j\otimes \pi^{\dagger}(d\langle p,\gamma \rangle),0) \\
=  (-\iota_{\eta}d(\dual{df}(\gamma)), j\otimes L_{\pi(p)}\gamma-j\otimes d\iota_{\pi(p)}(\gamma),0) \\
= (-\iota_{\eta}(\dual{df}(d\gamma)), j\otimes \iota_{\pi(p)}(d\gamma),0),
\end{multline*}
shows that the bracket $[\ , \ ]^{\prime}$ descends to an operation on $f^{++}(\mathcal{Q})$ which we will denoted by $[\ , \ ]$.

\begin{remark}
The construction of inverse image for Courant algebroids specializes to that \citep{LM} under suitable transversality assumptions.
\end{remark}

\subsection{Twisting by a 3-form}
Suppose that $\mathcal{Q}$ is a Courant algebroid on $X$ and $H$ is a closed 3-form on $X$. Recall that the $H$-twist of $\mathcal{Q}$, denoted $\mathcal{Q}_H$, is the Courant algebroid whose symmetric pairing and the anchor map coincide with the ones given on $\mathcal{Q}$ and the bracket is given by the $$[q,p]_{H}=[q,p]+\pi^{\dagger}({\iota_{\pi(p)}\iota_{\pi(q)}H}),$$ for any $q,p\in\mathcal{Q}$.

Suppose that $f\colon Y \to X$ is a map of manifolds. 

\begin{lemma}
\[
(f^\ii\mathcal{Q})_{f^*H} = f^\ii\mathcal{Q}_H .
\]
\end{lemma}
\begin{proof}
The underlying $\mathcal{O}_Y$-modules of both objects coincide. The bracket on $f^\ii\mathcal{Q}_H$ is given by the formula,
\begin{multline*}
[(\alpha, h\otimes q, \xi), (\beta, j\otimes p, \eta)] \\
= (-\iota_{\eta}d\alpha + L_{\xi}\beta + jdh\langle q,p \rangle, hj\otimes [q,p]_{H}-\iota_{\eta}dh\otimes q + L_{\xi}j\otimes p, [\xi,\eta])\\ =(-\iota_{\eta}d\alpha + L_{\xi}\beta + jdh\langle q,p \rangle, hj\otimes [q,p]-\iota_{\eta}dh\otimes q + L_{\xi}j\otimes p, [\xi,\eta])+(0,hj\otimes\pi^{\dagger}(\iota_{\pi(p)}\iota_{\pi(q)}H),0) \\
= [(\alpha, h\otimes q, \xi), (\beta, j\otimes p, \eta)] +(0,hj\otimes \pi^{\dagger}(\iota_{\pi(p)}\iota_{\pi(q)}H),0).
\end{multline*}
On the other hand, the bracket on $(f^\ii\mathcal{Q})_{f^*H}$ is given by the formula
\begin{multline*}
[(\alpha, h\otimes q, \xi), (\beta, j\otimes p, \eta)]_{f^*H}= [(\alpha, h\otimes q, \xi), (\beta, j\otimes p, \eta)]+ (\iota_{\xi}\iota_{\eta}f^{*}H,0,0) \\ = [(\alpha, h\otimes q, \xi), (\beta, j\otimes p, \eta)] +(0,\pi^{\dagger}(H(df(\xi),df(\eta),\;)),0)\\
 = [(\alpha, h\otimes q, \xi), (\beta, j\otimes p, \eta)] +(0,hj\otimes \pi^{\dagger}(\iota_{\pi(p)}\iota_{\pi(q)}H),0).
\end{multline*}
\end{proof}

\subsection{Connections}
Recall that a connection $\nabla$ on a Courant algebroid $\mathcal{Q}$ on $X$ is splitting $\nabla \colon \mathcal{T}_X \to \mathcal{Q}$  of the anchor map $\pi \colon \mathcal{Q} \to \mathcal{T}_X$ isotropic with respect to the symmetric pairing on $\mathcal{Q}$. We denote by $\mathcal{C}(\mathcal{Q})$ the sheaf of locally defined connections on $\mathcal{Q}$.

Suppose that $\mathcal{Q}$ is a Courant algebroid on $X$, and $f\colon Y \to X$ is a map of manifolds. A connection $\nabla$ on $Q$ induces a splitting $f^*(\nabla) \colon f^*\mathcal{T}_X \to f^*\mathcal{Q}$ of $f^*(\pi)$ and the splitting
\begin{equation}\label{splitting induced by connection}
\mathcal{T}_Y\times_{f^*\mathcal{T}_X}f^*(\nabla) \colon \mathcal{T}_Y \to \mathcal{T}_Y\times_{f^*\mathcal{T}_X}f^*\mathcal{Q}
\end{equation}
of $\mathcal{T}_Y\times_{f^*\mathcal{T}_X}f^*(\pi)$.

Let $f^\ii(\nabla) \colon \mathcal{T}_Y \to f^\ii\mathcal{Q}$ denote the composition
\[
\mathcal{T}_Y \xrightarrow{\eqref{splitting induced by connection}} \mathcal{T}_Y\times_{f^*\mathcal{T}_X}f^*\mathcal{Q} \to f^\ii\mathcal{Q} .
\]

\begin{lemma}\label{lemma: inverse image connections}
The map $f^\ii(\nabla)$ is a connection on $f^\ii\mathcal{Q}$.
\end{lemma}
\begin{proof}
For $\xi\in\mathcal{T}_Y$, $f^\ii(\nabla)(\xi)=(0,f^*(\nabla)(df(\xi)),\xi)$. Then
\[
\langle f^\ii(\nabla)(\xi),f^\ii(\nabla)(\eta)\rangle = 
\langle (0,f^*(\nabla)(df(\xi)),\xi), (0,f^*(\nabla)(df(\eta)),\eta)\rangle = 0.
\]
\end{proof}

Therefore, the map $f$ induces the morphism of sheaves
\begin{equation}\label{pull back of connections}
f^\ii \colon f^{-1}\mathcal{C}(\mathcal{Q}) \to \mathcal{C}(f^\ii\mathcal{Q}) .
\end{equation}

\subsection{Inverse image and linear algebra}
Suppose that $f\colon Y \to X$ is a map of manifolds and $\mathcal{A}\in\LA{\mathcal{O}_X}$.
\begin{lemma}\label{lemma: preservation Courant extensions}
Suppose that $\mathcal{Q}$ is a Courant extension of $\mathcal{A}$. Then, $f^\ii\mathcal{Q}$ is a Courant extension of $f^+\mathcal{A}$.
\end{lemma}
\begin{proof}
Follows from Proposition \ref{lemma: OExt and CExt} and Lemma \ref{lemma: inverse image preserves extensions}.
\end{proof}

Thus, the inverse image functor for Courant algebroids induces a functor
\begin{equation}\label{inverse image courant extension}
f^\ii \colon \CExt{\mathcal{A}} \to \CExt{f^+\mathcal{A}}
\end{equation}

\begin{proposition}
The functor \eqref{inverse image courant extension} is a morphism of $\mathbb{C}$-vector spaces in categories (and, in particular, of Picard groupoids).
\end{proposition}
\begin{proof}
Follows from Proposition \ref{lemma: tau and Q are linear} and Proposition \ref{lemma: inverse image is linear}.
\end{proof}

\subsection{Exact Courant algebroids}
Suppose that $\mathcal{Q}$ is an exact Courant algebroid on $X$. Recall (\citep{B}, 3.7) that the sheaf $\mathcal{C}(\mathcal{Q})$ of locally defined connections on $\mathcal{Q}$ together with the curvature map $c \colon \mathcal{C}(\mathcal{Q}) \to \Omega^{3,cl}_X$ is a $(\Omega^2_X \to \Omega^{3,cl}_X)$-torsor and the assignment $\mathcal{Q} \mapsto (\mathcal{C}(\mathcal{Q}), c)$ defines an equivalence $\ECA(X) \to (\Omega^2_X \to \Omega^{3,cl}_X)\mathtt{-torsors}$ of $\mathbb{C}$-vector spaces in categories.

Suppose that $f\colon Y \to X$ is a map of manifolds. The map \eqref{pull back of connections} is a morphism of torsors relative to the map $f^* \colon f^{-1}\Omega^2_X \to \Omega^2_Y$, hence induces the morphism of $\Omega^2_Y$-torsors
\begin{equation}\label{pull back connections torsor}
f^\ii \colon f^*\mathcal{C}(\mathcal{Q}) := \Omega^2_Y\times_{f^{-1}\Omega^2_X}f^{-1}\mathcal{C}(\mathcal{Q}) \to \mathcal{C}(f^\ii\mathcal{Q}).
\end{equation}

\begin{lemma}\label{lemma: pullback curvature}
The diagram
\[
\begin{CD}
f^{-1}\mathcal{C}(\mathcal{Q}) @>{f^\ii}>> \mathcal{C}(f^\ii\mathcal{Q}) \\
@V{f^*(c)}VV @VV{c}V \\
f^{-1}\Omega^{3,cl}_X @>{f^*}>>\Omega^{3,cl}_Y
\end{CD}
\]
is commutative.
\end{lemma}
\begin{proof}
By definition, for $\nabla\in\mathcal{C}(\mathcal{Q})$, the curvature $c(\nabla) \in \Omega^{3,cl}_X$ is defined by
\[
c(\nabla)(\xi,\eta)=[\nabla(\xi),\nabla(\eta)]-\nabla([\xi,\eta]),
\]
where $\xi, \ \eta\in\mathcal{T}_X$.

For $\xi, \ \eta\in\mathcal{T}_Y$, the computation
\begin{multline*}
c(f^\ii(\nabla))(\xi,\eta) = [f^\ii(\nabla(\xi)),f^\ii(\nabla(\eta))]-f^\ii(\nabla([\xi,\eta])) \\
= (0,[f^*(\nabla)(df(\xi)),f^*(\nabla)(df(\eta))],[\xi,\eta])-(0,f^*(\nabla)(df([\xi,\eta])),[\xi,\eta])\\
=(0,[f^*(\nabla)(df(\xi)),f^*(\nabla)(df(\eta))]-f^*(\nabla)(df([\xi,\eta])),0) \\
=((df)^{\vee}([f^*(\nabla)(df(\xi)),f^*(\nabla)(df(\eta))] - f^*(\nabla)(df([\xi,\eta]))),0,0) \\ 
= f^*(c(\nabla))(\xi,\eta)
\end{multline*}
implies the claim.
\end{proof}

\begin{proposition}
The diagram
\[
\begin{CD}
\ECA(X) @>{(\mathcal{C}, c)}>> (\Omega^2_X \to \Omega^{3,cl}_X)\mathtt{-torsors} \\
@V{f^\ii}VV @VV{f^*}V \\
\ECA(Y) @>{(\mathcal{C}, c)}>> (\Omega^2_Y \to \Omega^{3,cl}_Y)\mathtt{-torsors}
\end{CD}
\]
is commutative.
\end{proposition}
\begin{proof}
The map \eqref{pull back connections torsor} provides the isomorphism $f^*\circ\mathcal{C} \cong \mathcal{C}\circ f^\ii$. Lemma \ref{lemma: pullback curvature} says that it is a morphism of $(\Omega^2_Y \to \Omega^{3,cl}_Y)$-torsors.
\end{proof}

\subsection{Smooth descent for Courant algebroids}\label{subsection: Smooth localization for Courant algebroids}
Let $\CA$ denote the category with objects pairs $(X,\mathcal{Q})$, where $X$ is a manifold and $\mathcal{Q} \in \CA(X)$. A morphism $u \colon (Y,\mathcal{Q}^\prime) \to (X,\mathcal{Q})$ is a pair $u=(f, t)$, where $f \colon Y \to X$ is a map of manifolds and $t \colon \mathcal{Q}^\prime \to f^\ii\mathcal{Q}$ is a morphism in $\CA(Y)$. Lemma \ref{lemma: cfg lie algd} and the definition of the inverse image functor imply that the forgetful functor $(X,\mathcal{Q}) \mapsto X$ makes $\CA$ a prefibered category over $\Man$.

Let $\CA^\mathtt{lf}(X)$ denote the full subcategory of $\CA(X)$ with objects locally free of finite rank over $\mathcal{O}_X$. This is an example of the framework of \ref{subsection: general framework} with $\mathcal{P} = \CA$, $\mathcal{P}^\flat_X = \CA^\mathtt{lf}(X)$ and the functor of inverse image defined in \ref{subsection: inverse image courants}.

\begin{corollary}[of Theorem \ref{thm: smooth descent}]
$\CA^\mathtt{lf}$ has the smooth descent property.
\end{corollary}

\subsection{Dirac structures with support}
Suppose that $Z$ is a submanifold of $X$ and let $i\colon Z \to X$ denote the embedding. Let $\mathcal{Q}$ be a Courant algebroid on $X$ locally free of finite rank over $\mathcal{O}_X$. For an $\mathcal{O}_X$-module $\mathcal{E}$ and a submodule $\mathcal{F}\subset i^*\mathcal{E}$ we denote by $\widetilde{\mathcal{F}}$ the sub-module of $\mathcal{E}$ defined by the Cartesian square
\[
\begin{CD}
\widetilde{\mathcal{F}} @>>> \mathcal{E} \\
@VVV @VVV \\
i_*\mathcal{F} @>>> i_*i^*\mathcal{E}
\end{CD}
\]

\begin{definition}[\citep{AX,BIS}]\label{def: dirac with support}
A \emph{Dirac structure in $\mathcal{Q}$ supported on $Z$} is a sub-bundle $\mathcal{K}\subset i^*\mathcal{Q}$ which satisfies
\begin{enumerate}
\item $\mathcal{K}$ is maximal isotropic with respect to the restriction of the symmetric pairing;
\item $\mathcal{K}$ is mapped to $\mathcal{T}_Z$ under (the restriction of) the anchor map;
\item the sheaf $\widetilde{\mathcal{K}}$ is closed under the bracket on $\mathcal{Q}$.
\end{enumerate}
\end{definition}

We denote the collection of Dirac structures in $\mathcal{Q}$ supported on $Z$ by $\DirStr_Z(\mathcal{Q})$ and set $\DirStr(\mathcal{Q}) := \DirStr_X(\mathcal{Q})$.

\begin{remark}\label{remark: almost dirac}
The second condition in Definition \ref{def: dirac with support} is equivalent to $\mathcal{K}\subset \mathcal{T}_Z\times_{i^*\mathcal{T}_X} i^*\mathcal{Q}\subset i^*\mathcal{Q}$.
\end{remark}

Let $\dual{\mathcal{N}}_{Z\vert X}$ denote the conormal bundle defined by the exact sequence
\[
0 \to \dual{\mathcal{N}}_{Z\vert X} \to i^*\Omega^1_X \xrightarrow{\dual{di}} \Omega^1_Z \to 0 ,
\]
i.e. $\dual{\mathcal{N}}_{Z\vert X} = \ann(\mathcal{T}_Z)$.

Suppose that $\mathcal{K}$ is a Dirac structure supported on $Z$. In view of Remark \ref{remark: almost dirac} $\mathcal{K} + \pi^\dagger(\dual{\mathcal{N}}_{Z\vert X})$ is isotropic. Therefore, by maximality of $\mathcal{K}$, $\mathcal{K} = \mathcal{K} + \pi^\dagger(\dual{\mathcal{N}}_{Z\vert X})$, and $\pi^\dagger(\dual{\mathcal{N}}_{Z\vert X}) \subset \mathcal{K}$.

Recall that, by definition, $i^\ii\mathcal{Q} = \mathcal{T}_Z\times_{i^*\mathcal{T}_X} i^*\mathcal{Q}/\pi^\dagger(\dual{\mathcal{N}}_{Z\vert X})$. Let
\[
i^\ii\mathcal{K} := \mathcal{K}/\pi^\dagger(\dual{\mathcal{N}}_{Z\vert X}) .
\]

\begin{proposition}\label{prop: dirsupp bj dir}
{~}
\begin{enumerate}
\item $i^\ii\mathcal{K}$ is a Dirac structure in $i^\ii\mathcal{Q}$.

\item The assignment $\mathcal{K} \mapsto i^\ii\mathcal{K}$ defines a bijection between the set of almost Dirac structures in $\mathcal{Q}$ supported on $Z$ and the set of almost Dirac structures in $i^\ii\mathcal{Q}$.
\end{enumerate}
\end{proposition}
\begin{proof}
Given such a $\Dirac \subset  i^\ii\mathcal{Q}$, its pre-image in $i^+\mathcal{Q}$, i.e. $\Dirac + \pi^\dagger(\dual{\mathcal{N}}_{Z\vert X})$,  is a Dirac structure supported on $Z$.
\end{proof}

For $\phi\in\Hom_{\CA(X)}(\mathcal{Q}_1, \mathcal{Q}_2)$ the graph $\Gamma_\phi$ is a subsheaf of $\mathcal{Q}_1 \times \mathcal{Q}_2$. Since, by definition, $\phi$ induces the identity map on $\mathcal{T}_X$, it follows that $\Gamma_\phi\subset\mathcal{Q}_1 \times_{\mathcal{T}_X} \mathcal{Q}_2$.
Since, by definition, $\phi$ restricts to the identity map on $\Omega^1_X$, it follows that $\Gamma_\phi\cap(\Omega^1_X\times\Omega^1_X)$ is the diagonal. Therefore, the restriction of the map
\[
\mathcal{Q}_1 \times_{\mathcal{T}_X} \mathcal{Q}_2 \to \mathcal{Q}_1\dotplus\mathcal{Q}_2
\]
to $\Gamma_\phi$ is a monomorphism and (the image of) $\Gamma_\phi$ is a Dirac structure in $\mathcal{Q}_1\dotplus\mathcal{Q}_2$. The assignment $\phi \mapsto \Gamma_\phi$ defines a canonical map
\[
\Hom_{\CA(X)}(\mathcal{Q}_1, \mathcal{Q}_2) \to \DirStr(\mathcal{Q}_1\dotplus\mathcal{Q}_2^{op}) .
\]

\subsection{Courant algebroid morphisms (\citep{AX,BIS})}
Suppose that $f\colon Y \to X$ is a map of manifolds, $\mathcal{Q}_X \in \CA(X)$, $\mathcal{Q}_Y \in \CA(Y)$. Let $\pr_X \colon Y\times X \to X$ and $\pr_Y \colon Y\times X \to Y$ denote the projections; let $\gamma_f \colon Y \to Y\times X$ denote the graph embedding $y \mapsto (y, f(y))$.

The sheaf $\pr_Y^*\mathcal{Q}_Y\oplus\pr_X^*\mathcal{Q}^{op}_X$ is endowed with the canonical structure of a Courant algebroid on $Y\times X$ canonically isomorphic to $\pr_Y^\ii\mathcal{Q}_Y\dotplus\pr_X^\ii\mathcal{Q}^{op}_X$.

According to Proposition \ref{prop: dirsupp bj dir} a \emph{Courant algebroid morphism} (\citep{AX,BIS}) $\mathcal{K} \in \DirStr_{\gamma_f(Y)}(\pr_Y^\ii\mathcal{Q}_Y\dotplus\pr_X^\ii\mathcal{Q}^{op}_X)$ corresponds to the Dirac structure $\gamma_f(Y)^\ii\mathcal{K} \in \DirStr(\gamma_f^\ii(\pr_Y^\ii\mathcal{Q}_Y\dotplus\pr_X^\ii\mathcal{Q}^{op}_X)) \cong \DirStr(\mathcal{Q}_Y\dotplus f^\ii\mathcal{Q}^{op}_X)$ and there is a canonical map
\[
\Hom_{\CA(Y)}(\mathcal{Q}_Y, f^\ii\mathcal{Q}_X) \to \DirStr_{\gamma_f(Y)}(\pr_Y^\ii\mathcal{Q}_Y\dotplus\pr_X^\ii\mathcal{Q}^{op}_X) .
\]

\end{document}